\documentclass[11pt]{article}
\usepackage[utf8]{inputenc}
\usepackage[T1]{fontenc}

\usepackage{amsmath, amssymb, amsthm, amsfonts}
\usepackage{mathtools}
\usepackage{thm-restate}

\usepackage[dvipsnames,svgnames,x11names]{xcolor}
\definecolor{ForestGreen}{rgb}{0.1333,0.5451,0.1333}
\usepackage[colorlinks=true,linkcolor=ForestGreen,citecolor=blue]{hyperref}

\usepackage[margin=1in]{geometry}
\usepackage{graphics}
\usepackage{pifont}
\usepackage{tikz}
\usepackage{bbm}
\usepackage[T1]{fontenc}

\usetikzlibrary{arrows.meta}
\usepackage{environ}
\usepackage{framed}
\usepackage{url}
\usepackage{algorithm}
\usepackage[algo2e,linesnumbered,lined,ruled,boxed]{algorithm2e}
\usepackage{algpseudocode}
\usepackage[noabbrev,capitalize,nameinlink]{cleveref}
\crefname{equation}{}{}
\usepackage[labelfont=bf]{caption}
\usepackage{cite}
\usepackage{framed}
\usepackage[framemethod=tikz]{mdframed}
\usepackage{appendix}
\usepackage{graphicx}
\usepackage[textsize=tiny]{todonotes}
\usepackage{tcolorbox}
\allowdisplaybreaks[1]

\newcommand\remove[1]{}

\newtheorem{lemma}{Lemma}[section]
\newtheorem{theorem}{Theorem}
\newtheorem*{lemma*}{Lemma}

\newtheorem*{corollary*}{Corollary}

\theoremstyle{definition}

\newtheorem*{theorem*}{Theorem}
\newtheorem{definition}[lemma]{Definition}

\newtheorem*{rem*}{Remark}

\newtheorem{obs}{Observation}

\newcommand{\mc}{\mathcal}
\newcommand{\eps}{\varepsilon}
\newcommand{\R}{\mathbb{R}}

\newcommand{\E}{\mathop{\mathbb{E}}}
\renewcommand{\O}{\widetilde{O}}

\crefname{algocf}{Algorithm}{Algorithms}
\renewcommand{\l}{\langle}
\newcommand{\rr}{\rangle}
\newcommand{\supp}{\mathsf{supp}}
\newcommand{\Z}{\mathbb{Z}}

\newcommand{\bbC}{\mathbb{C}}

\newcommand{\A}{\Sigma}
\newcommand{\B}{\Gamma}
\newcommand{\C}{\Phi}

\renewcommand{\bar}{\overline}

\newcommand{\wt}{\widetilde}

\renewcommand{\bar}{\overline}

\newcommand{\F}{\mathbb{F}}
\newcommand{\dtv}{\mathrm{d}_{\mathrm{TV}}}
\newcommand{\dhj}{\mathsf{DHJ}}
\newcommand{\struct}{\mathsf{Struct}}
\newcommand{\bE}{\bar{E}}
\newcommand{\bbE}{\bar{E}'}
\renewcommand{\ge}{\geqslant}
\renewcommand{\le}{\leqslant}
\renewcommand{\geq}{\geqslant}

\begin{document}

\title{Reasonable Bounds for Combinatorial Lines of Length Three}

\author{Amey Bhangale\thanks{Department of Computer Science and Engineering, University of California, Riverside. Supported by the Hellman Fellowship award.}
	\and
	Subhash Khot\thanks{Department of Computer Science, Courant Institute of Mathematical Sciences, New York University. Supported by
		the NSF Award CCF-1422159, NSF CCF award 2130816, and the Simons Investigator Award.}
	\and
    Yang P. Liu\thanks{School of Mathematics, Institute for Advanced Study, Princeton, NJ. This material is based upon work supported by the National Science Foundation 
under Grant No. DMS-1926686}
    \and 
	Dor Minzer\thanks{Department of Mathematics, Massachusetts Institute of Technology. Supported by NSF CCF award 2227876 and NSF CAREER award 2239160.}}
\date{\vspace{-5ex}}
\clearpage\maketitle

\begin{abstract}
We prove that any subset $A \subseteq [3]^n$ with $3^{-n}|A| \ge (\log\log\log\log n)^{-c}$ contains a combinatorial line of length $3$, i.e., $x, y, z \in A$, not all equal, with $x_i=y_i=z_i$ or $(x_i,y_i,z_i)=(0,1,2)$ for all $i = 1, 2, \dots, n$. This improves on the previous best bound of $3^{-n}|A| \ge \Omega((\log^* n)^{-1/2})$ of [D.H.J. Polymath, Ann. of Math. 2012].
\end{abstract}

\pagenumbering{gobble}

\newpage

\setcounter{tocdepth}{2}
\tableofcontents

\normalsize
\pagebreak
\pagenumbering{arabic}

\section{Introduction}
The density Hales-Jewett theorem, proven by Furstenberg and Katznelson \cite{FK89,FK91}, asserts that for every positive integer $k$ and $\delta > 0$, that for sufficiently large $n$, a subset $S \subseteq \{0,1,\dots,k-1\}^n$ of size at least $\delta k^n$ (density at least $\delta$) must contain a \emph{combinatorial line} of length $k$. Here, we say that points $x^{(1)}, \dots, x^{(k)} \in S$ form a combinatorial line of length $k$ if not all $x^{(j)}$ are equal, and for each $i = 1, 2, \dots, n$, either $x^{(1)}_i = \dots = x^{(k)}_i$ or $(x^{(1)}_i, \dots, x^{(k)}_i) = (0, 1, \dots, k-1)$. Henceforth, we refer to this as the $\dhj[k]$ problem, and let $[k] := \{0,1,\dots,k-1\}$. This is the density version of the Hales-Jewett theorem \cite{HJ63}, which asserts for any positive integers $k$ and $r$, for all sufficiently large $n$ it holds that an $r$-coloring of $[k]^n$ contains a monochromatic combinatorial line. While the original proof of Furstenberg and Katznelson of the $\dhj[k]$ theorem was based on ergodic theory and thus gave ineffective bounds on $n$ in terms of $k$ and $\delta$, the Polymath project \cite{DHJ12} provided an elementary proof that gave effective bounds. For $k = 3$, which is the focus of this paper, \cite{DHJ12} proved that a subset $S \subseteq [3]^n$ with density $\delta$ contains a combinatorial line as long as $n \ge T(O(\delta^{-2}))$, where $T(m) = 2^{T(m-1)}$ and $T(0) = 1$, i.e., a tower of height $O(\delta^{-2})$. Put another way, a subset $S \subseteq [3]^n$ of density at least $\Omega((\log^* n)^{-1/2})$ contains a combiantorial line of length $3$.

The main result of this paper is an improvement of this bound to a tower of finite height.
\begin{theorem}
\label{thm:main}
There are constants $c, C > 0$ such that for all positive integers $n$ and subsets $|A| \subseteq [3]^n$ with $3^{-n}|A| \ge C(\log\log\log\log n)^{-c}$, there are $x, y, z \in A$, not all equal, forming a combinatorial line. In other words, for all $i = 1, \dots, n$, either $x_i = y_i = z_i$ or $(x_i,y_i,z_i) = (0,1,2)$.
\end{theorem}

\subsection{Connections to Other Problems}
\label{subsec:other}

The density Hales-Jewett problem has several connections to problems in extremal and additive combinatorics.
The van der Waerden \cite{vdw27} theorem says that for any positive integer $k$, a finite coloring of the positive integers contains a monochromatic arithmetic progression of length $k$ (a $k$-AP). Szemer\'{e}di's theorem \cite{Sz75} is the density analogue of van der Waerden's theorem, and says that for sufficiently large $n$ in terms of $\delta$ and $k$, that any density $\delta$ subset of $\{1, \dots, n\}$ contains a $k$-AP. In fact, Szemer\'{e}di's theorem follows from the $\dhj[k]$ theorem in the following way. Encode a point $x \in [k]^n$ as $\Pi(x) := \sum_{i=1}^n x_i k^i$. Now note that a combinatorial line corresponds exactly to a $k$-AP.
Another related problem is the corners problem, which asks to find in a dense subset of $[n] \times [n]$, three points that form an axis-aligned isosceles triangle: $(x, y)$, $(x+d, y)$, $(x, y+d)$. This was solved by Ajtai and Szemer\'{e}di \cite{AS74}. This can also be deduced from the $\dhj[k]$ theorem. Even more generally, the multi-dimensional Szemer\'{e}di theorem states that one can find even more complicated such patterns in dense subsets of $[n]^d$. This was also initially proven by Furstenberg and Katznelson \cite{FK78} by ergodic methods, but now a combinatorial proof is known via reduction to the $\dhj[k]$ theorem. The fact that so many problems can be reduced to $\dhj[k]$ may partly explain the difficulty in obtaining improved bounds for it.

While Szemer\'{e}di's original proof did not achieve great quantitative bounds, recent works have introduced several new tools broadly based on Fourier analysis that have greatly improved our understanding of Szemer\'{e}di's theorem, while providing significantly improved quantitative bounds. In \cite{Gowers98,Gowers01}, Gowers introduced higher-order Fourier analysis and the Gowers $U^k$-norms, which have played a crucial role in establishing analytic proofs and improved bounds for special cases of the multidimensional Szemer\'{e}di theorem and polynomial Szemer\'{e}di theorem \cite{Shk05,Shk06,GTZ11,GTZ12,GT12,GT17,Peluse18,Peluse19,Peluse20,HLY20,PP22,Peluse24,LSS23,Leng24,LSS24,LSS24b}, as well as several applications to number theory and beyond \cite{Green05prime,GT08,GT10}.

In terms of lower bounds, the best known lower bound for the $\dhj[k]$ problem is density $\Omega(\exp(-(\log n)^{1/\lceil \log_2 k \rceil}))$ \cite{DHJ09}, using ideas from Behrend's lower bound construction for $3$-AP free sets \cite{B46}. For the $k = 3$ case of Szemer\'{e}di's theorem, in a remarkable recent result, Kelley and Meka \cite{KM23} achieved an upper bound that nearly matches Behrend's lower bound. The bound has been further refined by Bloom and Sisask \cite{BS23}.

\subsection{Connections to the Study of CSPs}
\label{intro:csp}

Motivated by applications to understanding the approximability of satisfiable constraint satisfaction problems (CSPs), Bhangale, Khot, and Minzer \cite{BKM1,BKM2,BKM3,BKM4,BKM5,BKM3ap}, studied the following very general problem. Consider finite alphabets $\A_1, \dots, \A_k$ and a distribution $\mu$ supported on a subset of $\A_1 \times \dots \times \A_k$. If $1$-bounded functions $f_i: \A_i^n \to \bbC$ (so, $|f_i(x)| \le 1$ for all $x \in \A_i^n$) for $i = 1, \dots, k$ satisfy
\begin{align}
\left|\E_{(x_1,\dots,x_k) \sim \mu^{\otimes n}}\left[f_1(x_1) \dots f_k(x_k) \right] \right| \ge \eps, \label{eq:corr}
\end{align}
then what structure can we deduce about the functions $f_i$? This arose in the study of CSPs because a CSP is formally just a subset of $\A_1 \times \dots \times \A_k$, which may correspond to the support of the distribution $\mu$ being considered.

This generality also captures several problems in additive combinatorics. To see the connection with $\dhj[k]$, one can consider the setting where $\A_1 = \dots = \A_k = [k]$ and $\mu$ is supported on $(x,\dots,x)$ for $x \in [k]$ and $(0,1,\dots,k-1)$, and $f_i$ are indicators of a set $S \subseteq [k]^n$. Then \eqref{eq:corr} exactly measures the density of combinatorial lines with respect to the measure $\mu$. Similarly, many other problems in additive combinatorics can be captured in the same way.

In \cite{BKM4}, the following structural result is proved. Let $k=3$ and $\mu$ be pairwise-connected (the projection of the support of $\mu$ to any two coordinates forms a connected graph, see \cref{def:pairwise}), and $\mu$ has no Abelian embeddings into $\Z$. In this case, there is a finite group $H$ with size depending only on $\mu$, such that for each $i = 1, 2, 3$, $f_i$ correlates to some Fourier character over $H$ times a low-degree function. Equivalently, after randomly setting some coordinates of $f_i$ (see \cref{def:rr}), $f_i$ correlates to a Fourier character over $H$ on the remaining coordinates. This was used to give reasonable bounds for restricted $3$-APs over $\F_p^n$, whose common difference lies in the set $\{0,1,2\}^n$ always.

In the companion paper \cite{csp6}, we extend the result of \cite{BKM4} by removing the assumption that $\mu$ has no Abelian embeddings into $\Z$, and prove that each function $f_i$, after random restriction, correlates to a product function (see \cref{thm:3csp}) as long as $\mu$ is pairwise-connected. Similarly, this can be applied to give reasonable bounds for restricted $3$-APs over $\F_p^n$, but now with common difference lying in the set $\{0,1\}^n$. These results do not immediately imply any bounds for the $\dhj[3]$ problem, because the distribution $\mu$ supported on $(0,0,0)$, $(1,1,1)$, $(2,2,2)$, $(0,1,2)$ is not pairwise-connected.

To prove \cref{thm:main}, we take inspiration from the Shkredov's approach to the corners problem \cite{Shk06}, as well as the combinatorial proof of the $\dhj[k]$ theorem \cite{DHJ12}. As we discuss later, interpreting the corners problem (say over $\F_2^n \times \F_2^n$) in the context of \eqref{eq:corr} also produces a distribution $\mu$ which is not pairwise-connected. However, Shkredov proves strong bounds for the corners problem by a density increment strategy. Additionally, the Polymath paper \cite{DHJ12} gives analogies for several ingredients in the corners proof in the context of $\dhj[3]$ (though their argument is based on the corners proof of Ajtai and Szemer\'{e}di and not of Shkredov). At a high level, we manage to combine the new structural results for pairwise-connected correlations from the companion papers of the authors \cite{csp6,csp7} (see \cref{thm:3csp,thm:4csp}) with Shkredov's corners proof and \cite{DHJ12} to give reasonable bounds for the $\dhj[3]$ problem.

\paragraph{$\dhj[k]$ for $k\geq 4$:} 
it is conceivable that an approach in
the spirit of the current paper could 
be used to establish reasonable bounds 
for $\dhj[k]$ for general $k$. Such 
approach would require, at the very least, analogs of the results in~\cite{csp6,csp7} for pairwise-connected $k$-ary CSPs, 
which we plan to study in future works. 

\subsection{Preliminaries}
\label{subsec:prelim}

\paragraph{General notation.} Let $[n] = \{1, \dots, n\}$. For a subset $I \subseteq [n]$ we write $\bar{I} = [n] \setminus I$. We let $\bbC$ denote complex numbers. We let $\supp(\mu)$ denote the support of a distribution $\mu$.
For a finite set $\A$ we write $x \sim \A^n$ to denote sampling $x$ uniformly randomly from $\A^n$.

\paragraph{Connectivity.} We define what it means for a distribution to be pairwise-connected. Let $\mu_{ij}$ be the restriction of the distribution $\mu$ to the $(i, j)$ coordinates.
\begin{definition}[Pairwise-connected]
\label{def:pairwise}
We say that a distribution $\mu$ on $\A_1 \times \dots \times \A_k$ is pairwise-connected if for all $1 \le i < j \le k$ the support of $\mu_{ij}$ forms a connected graph over vertex set $\A_i \cup \A_j$.
\end{definition}

\paragraph{Random restriction.} All our methods and results heavily rely on random restrictions. A random restriction takes a subset of the coordinates, and randomly fixes their values. Below, $I \sim_{1-\alpha} [n]$ means that $I$ is a random subset of $[n]$, such that each $i \in [n]$ is included in $I$ with probability $1-\alpha$, and $z \sim \nu^I$ means that $z \in \A^I$ is such that each $z_i$ is i.i.d.~and distributed according to $\nu$.
\begin{definition}[Random restriction]
\label{def:rr}
Let $\mu$ be a distribution over $\A^n$ and let $\mu = (1-\alpha)\nu + \alpha \mu'$ for distributions $\nu, \mu'$. Then for a function $f: \A^n \to \bbC$, $I \sim_{1-\alpha} [n]$, and $z \sim \nu^I$, we define the random restriction $f_{I\to z}: \A^{[n] \setminus I} \to \bbC$ as $f_{I \to z}(x) := f(x, z)$.
\end{definition}

\section{Proof Outline}
\label{sec:overview}

\subsection{Shkredov's Corners Proof}
\label{subsec:corner}

We start by recalling the key ideas behind Shkredov's bound on sets that do not contain a corner \cite{Shk05,Shk06}. For simplicity we consider the finite field version of the problem, where we have a set $S \subseteq \F_2^n \times \F_2^n$. A corner consists of three points $(x,y)$, $(x+d, y)$, $(x,y+d) \in S$ such that $d \neq 0$. The standard approach to such problems is the density increment method pioneered by Roth \cite{Roth53} to solve the $k = 3$ case of Szemer\'{e}di's theorem. At a high level, this method tries to prove that either the set $S$ has about the expected number of corners, or the set $S$ has larger density on some structured subset (for example, a subspace). Then one can pass to this structured object and repeat the argument.

One type of structure that arises in the corners problem is that of combinatorial rectangles: this was observed by Ajtai and Szemer\'{e}di \cite{AS74}.
Let $\mu(S)$ be the density of $S$. Then the number of corners in $S$ is given by $\Pr_{x,y,d \in \F_2^n}\left[(x,y), (x, y+d), (x+d, y) \in S \right]$. When can this value deviate far from $\mu(S)^3$, which is what it would be if the events $(x, y) \in S$, $(x, y+d) \in S$, $(x+d, y) \in S$ were truly independent? A natural example is when $S \subseteq E_1 \times E_2$ for some $E_1 \subseteq \F_2^n$, $E_2 \subseteq \F_2^n$. Let $\mu' := \mu(S)/\mu(E_1 \times E_2)$ be the relative density of $S$ within $E_1 \times E_2$.
Then heuristically, the probability that $(x, y), (x, y+d), (x+d, y) \in S$ is approximately $\mu(E_1)^2\mu(E_2)^2(\mu')^3$, which is much larger than $\mu(S)^3 = \mu(E_1)^3\mu(E_2)^3(\mu')^3$. This heuristic holds as long as:
\begin{enumerate}
    \item The probability that $x, x+d \in E_1$ and $y, y+d \in E_2$ is approximately $\mu(E_1)^2 \mu(E_2)^2$.
    \item The probability that $(x, y)$, $(x, y+d)$, $(x+d, y) \in S$ conditioned on item 1 is about $(\mu')^3$.
\end{enumerate}
To complete the argument, one must understand the precise conditions under which these items hold. It is not too difficult to argue that item 1 holds as long as $E_1$ and $E_2$ are $U^2$-pseudorandom. In other words, all their nontrivial Fourier coefficients are small. Also, a (nontrivial) sequence of Cauchy-Schwarz manipulations proves that item 2 holds if $S$ does not admit a density increment onto a combinatorial subrectangle $E_1' \times E_2' \subseteq E_1 \times E_2$ of nonnegligible size.

From here, Shkredov's proof proceeds via a ``double descent'' kind of argument. The goal is to maintain $S \subseteq E_1 \times E_2$ such that the relative density of $S$ within $E_1 \times E_2$ increases whenever there are no corners.
Given such $S \subseteq E_1 \times E_2$, one first tries to find a sub-instance where $E_1$ and $E_2$ are $U^2$-pseudorandom. This is referred to as the ``uniformizing'' step. At a high level, this is done by partitioning $E_1$ and $E_2$ among subspaces (into smaller combinatorial rectangles) on which they have large Fourier coefficients, until most pieces in the partition are $U^2$-pseudorandom. By an averaging argument, there exists a piece on which $S$ still has large relative density.

Now, Shkredov argues that either the number of corners is approximately equal to what the heuristic above says, or $S$ admits a relative density increment onto a combinatorial subrectangle $E_1' \times E_2'$. This is proven by several applications of the Cauchy-Schwarz inequality. Of course, now $E_1', E_2'$ are not necessarily $U^2$-pseudorandom anymore, so one has to uniformize again, and repeat the argument.

\subsection{Dictionary between Corners and $\dhj[3]$}
\label{subsec:dict}

The work \cite{DHJ12} gives a combinatorial proof of the density Hales-Jewett theorem by drawing analogies from the corners proof of Ajtai and Szemer\'{e}di \cite{AS74}. Here we describe how each aspect of the corners proof described in \cref{subsec:corner} can be adapted to the $\dhj[3]$ setting.

One immediate difference in the $\dhj[3]$ setting is that it is not clear under what measure we should be counting triples $(x, y, z) \in S$ that form a combinatorial line. \cite{DHJ12} introduces the \emph{equal slices measure} under which the counts of combinatorial lines is sensibly defined. In this work, we take a slightly different approach that is similar in spirit. Regardless, precisely defining the equal slices measure or our variant is not important for this outline -- it is only important that there is some sensible way to count combinatorial lines.

Let $S \subseteq [3]^n$ be the set on which we wish to find combinatorial lines. \cite{DHJ12} proves that if the count of combinatorial lines is far off from what is expected, then $S$ has increased relative density on the intersection of a $02$-insensitive and $01$-insensitive set (see also a simpler proof from~\cite{DKT14} that avoids the equal slice measure and works with the uniform distribution throughout). Here, a $02$-insensitive set is a subset $E_1 \subseteq [3]^n$ satisfying that if $x, y \in [3]^n$ agree on all $i \in [n]$ where $x_i = 1$ or $y_i = 1$, then $x \in E_1$ if and only if $y \in E_1$. In other words, only the locations of the $1$'s in $x$ determines whether $x \in E_1$. A $01$-insensitive set (which we call $E_2$) is defined similarly. We refer to the intersection of $E_1$ and $E_2$ as their \emph{disjoint product} (see \cref{def:disjprod}) and denote it as $E_1 \boxtimes E_2$. This is the first entry of the dictionary: $E_1 \boxtimes E_2$ corresponds to the combinatorial rectangles of the corners proof.

From here, \cite{DHJ12} argues that $E_1 \boxtimes E_2$ can be split into combinatorial subspaces (i.e., sets isomorphic to $[3]^{n'}$ for $n' \le n$), and thus one obtains a density increment of $S$ onto a combinatorial subspace, and the argument can be repeated. In \cite{DHJ12}, $n' \le \log n$ due to an application of the multidimensional Szemer\'{e}di theorem (for $2$ points), and thus the final density bound achieved on $S$ is $(\log^* n)^{-c}$. Similarly, our proof will also use combinatorial subspaces, but for larger $n' \ge n^c$.

\subsection{Pseudorandomness Notion}
\label{subsec:pseudo}
The main missing ingredient from fully translating Shkredov's proof to the $\dhj[3]$ setting is the correct notion of pseudorandomness. It turns out that the correct notion is what we call \emph{product-pseudorandomness} (see \cref{def:prodpseudo}). While $U^2$ pseudorandomness of a function $f$ says that $|\l f, \prod_{i=1}^n \chi_i\rr|$ is small for any Fourier characters $\chi_1, \dots, \chi_n$, product function pseudorandomness is a stronger notion which says that $|\l f, \prod_{i=1}^n P_i \rr|$ is small for any choice of $1$-bounded functions $P_1, \dots, P_n: [3] \to \bbC$, where $1$-bounded means that $|P_i(x)| \le 1$ for all $x \in [3]$.
In fact, in our definition we also require that random restrictions of $f$ down to even $n^{1/4}$ coordinates correlate to product functions $\prod_{i=1}^n P_i$ with only negligible probability.
This notion of pseudorandomness is motivated by the inverse theorems in \cite{csp6,csp7} (see \cref{thm:3csp,thm:4csp}) which say that if functions $f, g, h$ have large $3$-wise correlation over $\mu^{\otimes n}$ for a pairwise-connected distribution $\mu$, then each of $f$, $g$, $h$ is not product-pseudorandom.

It remains to discuss how to put all these pieces together. Similarly to the corners proof, our goal is to increment the relative density of $S$ within $E_1 \boxtimes E_2$ while there is no combinatorial line. 
We design a uniformization algorithm for our notion of product function pseudorandomness -- this essentially amounts to understanding how to density increment on sets that correlate to some product function. This is done in \cref{sec:uniform}, and is based on the density increment for restricted 3-progressions in \cite{BKM3ap,csp6}.
Finally we argue that if the number of combinatorial lines (under the proper measure) deviates from what is expected, then $S$ has a relative density increment onto a ``subrectangle'' $E_1' \boxtimes E_2' \subseteq E_1 \boxtimes E_2$. This is done in \cref{sec:equal,sec:ineq}.

It is worth discussing where product pseudorandomness enters the picture, more precisely. For example, we wish to argue that $E_1 \boxtimes E_2$ behaves like a combinatorial rectangle. One way to formalize this is that we want to have $\mu(E_1 \boxtimes E_2) \approx \mu(E_1)\mu(E_2)$, i.e., that the events $x \in E_1$ and $x \in E_2$ for $x \in [3]^n$ behave independently. In \cref{lemma:me1e2} we argue that product pseudorandomness of $E_1$ and $E_2$ suffices for this. In addition, throughout the proof we require more complex correlations of a similar type (e.g., \cref{lemma:eyz,lemma:4cspy,lemma:omega,lemma:smallmeasure}) which all leverage product pseudorandomness. We believe that any form of Fourier pseudorandomness would not suffice to prove these statements.

\subsection{Choosing Parameters}
\label{sec:param}

Throughout the paper we will use the following parameters and relationships between them.

The parameter $\alpha$ will be the relative density of the set $S$ within the disjoint product $E_1 \boxtimes E_2$.
The density of $E_1$ and $E_2$ within $[3]^n$ will be $\delta_1$ and $\delta_2$ respectively. These will decrease during density increments, but will always stay above $\exp(-\alpha^{-C})$ (say, for $C = 20$).

The parameter $\gamma$ will be a pseudorandomness parameter. We require that $\log(1/\gamma)^{-1} < (\alpha\delta_1\delta_2)^C$ for a large constant $C$. Thus, $\gamma \le \exp(-\exp(\alpha^{-C}))$.

During a uniformization step, $n$ will drop to about $n^{\exp(-\gamma^{-C})}$ for a constant $C$ (see \cref{thm:uniform}). This informs our choice of $\alpha$ ultimately, because we need
\[ \exp(-\gamma^{-C}) \le (\log n)^c, \enspace \text{ so } \alpha \ge \Omega((\log\log\log\log n)^{-c}), \] by our choice of $\gamma$.

Finally, we will set the parameter $\eta = \exp(-\sqrt{\log n})$, where $n$ is the original dimension. Throughout the whole procedure, $\eta$ will be much smaller than any of $\alpha, \delta_1, \delta_2$, and $\gamma$. The parameter $\eta$ arises sometimes when we are dealing with distributions $\nu$ on $[3]$ satisfying $|\nu(x) - 1/3| \le \eta/\sqrt{n}$. In other words, $\dtv(\mu^{\otimes n}, \nu^{\otimes n}) \le O(\eta)$, where $\mu$ is the uniform distribution on $[3]$.

\section{Correlations to Product Functions}
\label{sec:correlation}

\subsection{Product Pseudorandomness}
\label{subsec:prodpseudo}
We start by formally introducing the main pseudorandomness notion considered in this paper.
\begin{definition}[$(n',\gamma)$-product pseudorandomness]
\label{def:prodpseudo}
Let $\A$ be a finite set, $\mu$ be a distribution on $\A$, and $n$ a positive integer. For $n' \le n$ and $\gamma > 0$ we say that a $1$-bounded function $f: \A^n \to \bbC$ is $(n',\gamma)$-product pseudorandom if for any $\delta \in [n'/n, 1]$, the probability that a random restriction of $f$ down to $\bar{I} \sim_{\delta} [n]$ $\gamma$-correlates to a product function is less than $\gamma$. Precisely,
\[ \Pr_{I \sim_{1-\delta} [n], z \sim \mu^{\otimes I}}\left[\exists \{P_i: \A \to \bbC, \|P_i\|_\infty \le 1\}_{i \in \bar{I}} \enspace \text{ with } \enspace \Big|\E_{x \sim \mu^{\otimes \bar{I}}}\Big[f_{I\to z}(x) \prod_{i \in \bar{I}} P_i(x_i) \Big]\Big| \ge \gamma \right] < \gamma. \]
\end{definition}
Informally, another interpretation of product pseudorandomness is that $f$ does not correlate to a function of the form $LP$ where $L$ has \emph{degree} at most $\O(n/n')$, and $P$ is a product function. This can potentially be formalized by proving a restriction inverse theorem akin to \cite{BKM3,BKM4}, but it is more convenient to just work with \cref{def:prodpseudo} in the present work.

It is worth remarking that one can change the measure that a function $f$ is product pseudorandom with respect to by applying random restrictions. Informally, if $f$ is product pseudorandom with respect to $\mu$ and $\mu = \beta \nu + (1-\beta) \nu'$, then applying a random restriction $I \sim_{1-\beta} [n]$ and $z \sim (\nu')^{\otimes I}$, the function $f_{I \to z}$ will be product pseudorandom against $\nu$. A more formal and general statement appears in \cref{lemma:projpseudo}.

\subsection{Inverse Theorems for CSPs}
\label{subsec:csp}

In this section we discuss a significant generalization of the $U^2$ Gowers-norm inverse theorem that we require for our proof. The following theorem was proven in \cite[Theorems 1, 2]{csp6}.
\begin{theorem}
\label{thm:3csp}
Let $\mu$ be a pairwise-connected distribution over $\A \times \B \times \C$ for finite sets $\A, \B, \C$ on which the probability of each atom is at least $\alpha$. If $1$-bounded functions $f: \A^n \to \bbC$, $g: \B^n \to \bbC$, $h: \C^n \to \bbC$ satisfy that
\[ \left|\E_{(x,y,z) \sim \mu^{\otimes n}}[f(x)g(y)h(z)] \right| \ge \eps, \] then there is a constant $\gamma := \gamma(\alpha,\eps) > 0$ and distribution $\nu$ such that $\mu = (1-\gamma)\nu + \gamma U$, where $U$ is uniform over $\A$, such that:
\[ \Pr_{I \sim_{1-\gamma} [n], z \sim \nu^I}\left[\exists \{P_i: \A \to \bbC, \|P_i\|_\infty \le 1\}_{i \in \bar{I}} \enspace \text{ with } \enspace \Big|\E_{x \sim \A^{\bar{I}}}\Big[f_{I\to z}(x) \prod_{i \in \bar{I}} P_i(x_i) \Big]\Big| \ge \gamma \right] \ge \gamma. \]
Quantitatively, $\gamma(\eps) \ge \exp(-\eps^{-O_{\alpha}(1)})$.
\end{theorem}
Put another way, the functions $f$, $g$, $h$ are not product pseudorandom.
To see the connection to the $U^2$ inverse theorem, one can consider when $\A = \B = \C = H$ for some Abelian group $H$ and $\mu$ is uniform over $(x, x+d, x+2d)$ for $x, d \in H$. 

In this present work, we require an inverse theorem for certain $4$-ary distributions. This is done by reducing to the $3$-ary inverse theorem in \cref{thm:3csp}. The analogy is that the $U^2$ Gowers norm is over $4$-ary distributions $(x, x+a, x+b, x+a+b)$, but its inverse structure behaves the same as for $3$-APs. To formally state this theorem, we need to define what it means for a distribution to be connected.

\begin{definition}
\label{def:embed}
A subset $S \subseteq \A_1 \times \dots \times \A_k$ is said to be connected if for every $a = (a_1, a_2, \ldots, a_k)\in S$ and $b = (b_1, b_2, \ldots, b_k) \in S$, there exists a sequence $(a, c^{(1)}, \ldots, c^{(t)}, b)$ for some $t\geq 1$ such that $c^{(i)}\in S$ for all $i$ and every consecutive pair of tuples in the sequence differs in at most one coordinate. A distribution $\mu$ on $\A_1 \times \dots \times \A_k$ is said to be connected if the support of $\mu$ is connected.
\end{definition}
The following theorem is from \cite[Lemma 1.5]{csp7}.

\begin{theorem}
\label{thm:4csp}
Let $\mu$ be a distribution over $\A_1 \times \A_2 \times \A_3 \times \A_4$ such that the probability of each atom is at least $\alpha$ and the distribution $\mu_{123}$ is connected. There is a constant $\gamma := \gamma(\alpha,\eps) > 0$ such that if $1$-bounded functions $f_i: \A_i^n \to \bbC$ for $i = 1, 2, 3, 4$ satisfy 
\[ \left|\E_{(x_1,x_2,x_3,x_4) \sim \mu^{\otimes n}}[f_1(x_1)f_2(x_2)f_3(x_3)f_4(x_4)] \right| \ge \eps, \]
then a random restriction of $f_1$ correlates to a product function, i.e.,
\[ \Pr_{I \sim_{1-\gamma} [n], z \sim \mu_1^I} \left[\exists \{P_i: \A_1 \to \bbC, \|P_i\|_\infty \le 1\}_{i \in \bar{I}} \enspace \text{ with } \enspace \Big|\E_{x \sim \mu_1^{\bar{I}}}\Big[(f_1)_{I\to z}(x) \prod_{i \in \bar{I}} P_i(x_i) \Big]\Big| \ge \gamma \right] \ge \gamma. \]
Quantitatively, $\gamma \ge \exp(-\eps^{-O_{\alpha}(1)})$. 
\end{theorem}

\section{Inductive Structure}
\label{sec:struct}

Let $\mu$ be a measure on $[3]^n$. Let $\pi_1: \{0, 1, 2\}^n \to \{0, 1\}^n$ be defined as $\pi_1(x)_i = 1$ if and only if $x_i = 1$. Similarly, $\pi_2: \{0, 1, 2\}^n \to \{0, 2\}^n$ is defined as $\pi_2(x)_i = 2$ if and only if $x_i = 2$.

The measure of $E_1$ is denoted as
\[ \mu(E_1) := \mu(\{x \in [3]^n : \pi_1(x) \in E_1\}). \] $\mu(E_2)$ is defined similarly. By abuse of notation, we let $\mu$ denote the distribution on $\{0, 1\}^n$ (respectively $\{0,2\}^n$) distributed as $\pi_1(x)$ (resp. $\pi_2(x)$) for $x$ distributed according to $\mu$ in $[3]^n$.

We now define the disjoint product of subsets $E_1 \subseteq \{0,1\}^n$ and $E_2 \subseteq \{0,2\}^n$ which serves as our analogy to combinatorial rectangles.
\begin{definition}[Disjoint product]
\label{def:disjprod}
Let $E_1 \subseteq \{0, 1\}^n$ and $E_2 \subseteq \{0, 2\}^n$. The disjoint product $E_1 \boxtimes E_2 \subseteq \{0,1,2\}^n$ is defined as:
\[ E_1 \boxtimes E_2 := \{x \in [3]^n : \pi_1(x) \in E_1 \text{ and } \pi_2(x) \in E_2 \}. \]
\end{definition}
Our density increment proceeds by maintaining a subset $S \subseteq E_1 \boxtimes E_2$ whose relative density increases over steps of the density increment.
\begin{definition}[Structure]
\label{def:struct}
We define $\struct_{\alpha}$ to consist of triples $(S,E_1,E_2)$ satisfying $E_1 \subseteq \{0, 1\}^n$, $E_2 \subseteq \{0, 2\}^n$ with $S \subseteq E_1 \boxtimes E_2$ and for $\mu$ uniform over $[3]^n$:
\begin{enumerate}
    \item $\mu(S) \ge \alpha \cdot \mu(E_1 \boxtimes E_2)$.
    \item Let $\delta_1 = \mu(E_1)$ and $\delta_2 = \mu(E_2)$. The functions $1_{E_1}-\delta_1: \{0,1\}^n \to \R$ and $1_{E_2}-\delta_2: \{0,2\}^n \to \R$ are $(n',\gamma)$-product pseudorandom for $n' = n^{1/4}$ and $\gamma := \gamma(\alpha,\delta_1,\delta_2)$ chosen later.
\end{enumerate}
\end{definition}

At this point it is worth noting that the pseudorandomness condition implies that $\mu(E_1 \boxtimes E_2) \approx \mu(E_1)\mu(E_2)$. In fact, this holds as long as the measures of $E_1$ and $E_2$ do not change significantly under random restriction -- this was previously observed by Austin \cite[Lemma 6.1]{Austin11} in a different language.
\begin{lemma}
\label{lemma:me1e2}
Let $\mu$ be a distribution on $[3]$ with mass $\Omega(1)$ on each atom. If $E_1, E_2$ are $(n^{0.99},\gamma)$ product pseudorandom with respect to $\mu$ then 
\[ \left|\mu(E_1 \boxtimes E_2) - \mu(E_1)\mu(E_2)\right| \le 2\gamma^{1/2}. \]
\end{lemma}
\begin{proof}
Let $\bar{E_1} := 1_{E_1} - \mu(E_1)$ and $\bar{E_2} := 1_{E_2} - \mu(E_2)$. Then
\begin{align*}
\mu(E_1 \boxtimes E_2) &= \E_{x \sim \mu^{\otimes n}}\left[1_{E_1}(x)1_{E_2}(x) \right] = \mu(E_1)\mu(E_2) + \E_{x \sim \mu^{\otimes n}}\left[\bar{E_1}(x) \bar{E_2}(x) \right] \\
&= \mu(E_1)\mu(E_2) + \E_{x \sim \mu^{\otimes n}}\left[\bar{E_1}(\pi_1(x)) \bar{E_2}(\pi_2(x)) \right].
\end{align*}
Assume for contradiction that $|\E_{x \sim \mu^{\otimes n}}\left[\bar{E_1}(x) \bar{E_2}(x) \right]| > 2\sqrt{\gamma}$. Then by Cauchy-Schwarz:
\begin{align*}
4\gamma &\le \E_{\pi_2(x)} \left|\E_{x : \pi_2(x)} \bar{E_1}(\pi_1(x)) \right|^2 = \E_{\pi_2(x)} \E_{x,x' : \pi_2(x)} \bar{E_1}(\pi_1(x))\bar{E_1}(\pi_1(x')),
\end{align*}
where the notation means that we first sample $\pi_2(x)$ for $x \sim \mu^{\otimes n}$, and then sample $x, x'$ conditioned on $\pi_2(x)$. Because $\mu$ has full support on $[3]$, one can check that $(\pi_1(x), 
\pi_1(x'))$ has full support on $(0,0), (0,1), (1,0), (1,1)$. Let $\nu$ be the distribution on $(\pi_1(x), \pi_1(x'))$ and let $\nu = \beta U + (1-\beta)\nu'$ where $U$ is uniform over $\{0,1\} \times \{0,1\}$, and $\beta = \Omega(1)$. Then we conclude that
\begin{align*}
4\gamma &\le \E_{I \sim_{1-\beta} [n]} \E_{(z,z') \sim (\nu')^{\otimes I}} \E_{(x,x') \sim U^{\otimes \bar{I}}} \left[(\bar{E_1})_{I \to z}(x) (\bar{E_1})_{I \to z'}(x') \right] \\
&= \E_{I \sim_{1-\beta} [n]} \E_{(z,z') \sim (\nu')^{\otimes I}} \left(\E_{x \sim U^{\otimes \bar{I}}}\left[(\bar{E_1})_{I \to z}(x) \right] \right) \left( \E_{x \sim U^{\otimes \bar{I}}}\left[(\bar{E_1})_{I \to z'}(x) \right] \right) \\
&\le \E_{I \sim_{1-\beta} [n]} \E_{(z,z') \sim (\nu')^{\otimes I}} \left|\E_{x \sim U^{\otimes \bar{I}}}\left[(\bar{E_1})_{I \to z}(x) \right] \right|.
\end{align*}
By an averaging argument and the fact that $\beta n \ge n^{0.99}$, this contradicts that $\bar{E_1}$ is $(n^{0.99}, \gamma)$ product pseudorandom.
\end{proof}

\section{Moving between Measures}
\label{sec:equal}
In this section we prove that if $S$ contains no combinatorial lines, then some particular $3$-wise correlation over a product distribution whose atoms are at least $\Omega(1)$ in each coordinate is large. Showing this involves picking a suitable measure under which to count combinatorial lines.

\subsection{Lower Bounding the Expected Count}
\label{sec:lowerbound}

Let $\nu$ be a measure on $[3]^n \times [3]^n \times [3]^n$, such that for all $i \in [n]$, $\nu_i$ is supported on triples $(0,0,0), (1,1,1), (2,2,2), (0,1,2)$. Let $S \subseteq E_1 \boxtimes E_2 \subseteq [3]^n$ with relative density $\alpha$ in $E_1 \boxtimes E_2$, $E_1$ has density $\delta_1$, and $E_2$ has density $\delta_2$ (here, we are being a bit informal in that we are not specifying the density we are using here, but it may not be uniform). The following discussion holds for any $\nu$, but we will pick convenient ones to work with soon.

We can count the density of combinatorial lines contained in $S$ with respect to $\nu$. This is:
\begin{align*}
    \E_{(x,y,z) \sim \nu}[1_S(x)1_S(y)1_S(z)] &= \E_{(x,y,z) \sim \nu}[(1_S(x)-\alpha 1_{E_1\boxtimes E_2}(x))1_S(y)1_S(z)] \\ ~&+ \alpha \E_{(x,y,z) \sim \nu}[1_{E_1\boxtimes E_2}(x) 1_S(y)1_S(z)].
\end{align*}
On the right hand side of the equation, the first term is the ``error'' term, and the second is the ``main term''. Studying the main term, one can note that $1_{E_1\boxtimes E_2}(x) 1_S(y)1_S(z) = 1_S(y)1_S(z)$ in fact, because $\pi_2(y) = \pi_2(x)$ and $\pi_1(z) = \pi_1(x)$ for all $(x,y,z) \in \supp(\nu)$. Thus, the main term can we written more succinctly as: $\alpha  \E_{(x,y,z) \sim \nu}[1_S(y)1_S(z)]$. We would like to pick a measure $\nu$ on which we can lower bound this quantity. This is essentially asking for a measure $\nu$ on which the counting version of $\dhj[2]$ holds. A key observation made in \cite[Theorem 3.1]{DHJ12} is that this holds when $\nu$ is the \emph{equal-slices measure}. We take a somewhat different approach here which is similar in spirit, but allows us to work only with product distributions throughout the arguments.

Let $\eta$ be sufficiently small (depending on $\alpha, \delta_1, \delta_2$), and let $K \in \Z_{\ge1}, \eta' \in \R_{>0}$ satisfy $K\eta' \le \eta/100$. Define the following distribution over sequences in $([3]^n)^{K+1}$.
\begin{definition}
\label{def:seqdist}
Define the distribution $\mc{D}$ over $(y^{(0)}, \dots, y^{(K)}) \in ([3]^n)^{K+1}$ as follows. $y^{(0)}$ is drawn uniformly in $[3]^n$. For $i = 0, \dots, K-1$, define for $t \in [n]$ independently:
\[ y^{(i+1)}_t :=
\begin{cases}
y^{(i)}_t, & \text{ if } y^{(i)}_t \in \{0, 2\}, \\
1 \text{ w.p. } 1-\frac{\eta'}{\sqrt{n}}, & \text{ if } y^{(i)}_t = 1, \\
2 \text{ w.p. } \frac{\eta'}{\sqrt{n}}, & \text{ if } y^{(i)}_t = 1.
\end{cases}\]
\end{definition}
In other words, to get $y^{(i+1)}$ from $y^{(i)}$, for each index $t \in [n]$ where $y^{(i)}_t = 1$, we change it to a $2$ with probability $\eta'/\sqrt{n}$. We now record an observation about what the distribution of $y^{(i)}$ and joint distribution $(y^{(i)}, y^{(j)})$ look like.
\begin{obs}
\label{obs:joint}
For all $i = 0, \dots, K$, $y^{(i)}$ is distributed as $(\nu^{(i)})^{\otimes n}$ for a distribution $\nu^{(i)}$ on $[3]$ satisfying that $\nu^{(i)}(0) = 1/3$ and $|\nu^{(i)}(x) - 1/3| \le \eta/\sqrt{n}$ for $x \in \{1, 2\}$. For all pairs $0 \le i < j \le K$, $(y^{(i)}, y^{(j)})$ is distributed as $(\xi^{(ij)})^{\otimes n}$ for some distribution $\xi^{(ij)} \in [3] \times [3]$ supported on $(0,0),(1,1),(2,2)$, and $(1,2)$ satisfying that:
\begin{enumerate}
    \item $\xi^{(ij)}((0,0)) = 1/3$ and $|\xi^{(ij)}((x,x)) - 1/3| \le \eta/\sqrt{n}$ for $x \in \{1, 2\}$.
    \item $\xi^{(ij)}((1,2)) \ge \frac{\eta'}{10\sqrt{n}}$. \label{item:lower12}
\end{enumerate}
\end{obs}
The first item is actually redundant with the claim about the distribution of $y^{(i)}$, but we list it for completeness. Finally, we argue that there is some pair $(i, j)$ under which $\E_{(y,z) \sim (\xi^{(ij)})^{\otimes n})}[1_S(y)1_S(z)]$ can be lower bounded.
\begin{lemma}
\label{lemma:mainterm}
Let $\mu$ be uniform on $[3]^n$. For every subset $S \subseteq [3]^n$, there is $0 \le i < j \le K$ such that
\[ \E_{(y,z) \sim (\xi^{(ij)})^{\otimes n}}\Big[1_S(y)1_S(z)\Big] \ge \mu(S)^2 - 6\eta - \frac{\mu(S)}{K}.\]
\end{lemma}
\begin{proof}
Note that
\begin{align*}
    \sum_{i \neq j} 1_S(y^{(i)}) 1_S(y^{(j)}) = \left(\sum_{i=0}^K 1_S(y^{(i)})\right)^2 - \sum_{i=0}^K 1_S(y^{(i)}).
\end{align*}
Also, because $\dtv(\nu{(i)}, \mu) \le 2\eta$ by \cref{obs:joint} for all $i = 0, \dots, K$ we know that
\[ (K+1)(\mu(S) + 2\eta) \ge \E_{(y^{(0)}, \dots, y^{(K)}) \sim \mc{D}} \left[\sum_{i=0}^K 1_S(y^{(i)}) \right] \ge (K+1)(\mu(S) - 2\eta). \] Thus by Cauchy-Schwarz
\[ \sum_{i \neq j} \E_{(y^{(0)}, \dots, y^{(K)}) \sim \mc{D}} \left[1_S(y^{(i)}) 1_S(y^{(j)}) \right] \ge (K+1)^2(\mu(S)-2\eta)^2 - (K+1)(\mu(S)+2\eta). \]
By averaging, there is a pair $(i, j)$ such that
\begin{align*} 
\E_{(y^{(0)}, \dots, y^{(K)}) \sim \mc{D}} \left[1_S(y^{(i)}) 1_S(y^{(j)}) \right] &\ge \frac{(K+1)^2(\mu(S)-2\eta)^2 - (K+1)(\mu(S)+2\eta)}{(K+1)K} \\
&\ge (\mu(S) - 2\eta)^2 - \frac{\mu(S) + 2\eta}{K} \ge \mu(S)^2 - 6\eta - \frac{\mu(S)}{K}.
\end{align*}
This is exactly the desired conclusion, given the definition of $\xi^{(ij)}$ in \cref{obs:joint}.
\end{proof}
Let $\xi^{(ij)'}$ be the unique distribution supported on $(0,0,0), (1,1,1), (2,2,2), (0,1,2)$ such that $(\xi^{(ij)'})_{yz} = \xi^{(ij)}$. We will use $\xi^{(ij)'}$ to argue that if a set $S \subseteq E_1 \boxtimes E_2$ has no combinatorial lines, then some specific $3$-wise correlation is large.
\begin{lemma}
\label{lemma:error}
Let $S \subseteq E_1 \boxtimes E_2$ be such that $\mu(E_1 \boxtimes E_2) = \delta$, $\mu(S) \ge \alpha\delta$, and $S$ contains no combinatorial line. For $K \ge 100\mu(S)^{-1}$ and $\eta \le \mu(S)^2/100$ there is some $0 \le i < j \le K$ such that
\[ \left|\E_{(x,y,z) \sim (\xi^{(ij)'})^{\otimes n}}\Big[(1_S - \alpha 1_{E_1 \boxtimes E_2})(x) 1_S(y) 1_S(z) \Big] \right| \ge \alpha^3\delta^2/2. \]
\end{lemma}
\begin{proof}
By \cref{lemma:mainterm} there are $0 \le i < j \le K$ such that
\[ \E_{(y,z) \sim (\xi^{(ij)})^{\otimes n}}\Big[1_S(y)1_S(z)\Big] \ge \mu(S)^2 - 6\eta - \frac{\mu(S)}{K} \ge \frac34\mu(S)^2 \ge \frac34 \alpha^2\delta^2. \]
If $S$ has no combinatorial line, then
\begin{align*}
\exp(-\Omega(\eta'\sqrt{n})) &\ge \E_{(x,y,z) \sim (\xi^{(ij)'})^{\otimes n}}\Big[1_S(x) 1_S(y) 1_S(z)\Big] \\
&= \E_{(x,y,z) \sim (\xi^{(ij)'})^{\otimes n}}\Big[(1_S - \alpha 1_{E_1\boxtimes E_2})(x) 1_S(y) 1_S(z)\Big] \\ 
&~+ \alpha \E_{(x,y,z) \sim (\xi^{(ij)'})^{\otimes n}}\Big[1_{E_1 \boxtimes E_2}(x) 1_S(y) 1_S(z)\Big] \\
&\ge \E_{(x,y,z) \sim (\xi^{(ij)'})^{\otimes n}}\Big[(1_S - \alpha 1_{E_1\boxtimes E_2})(x) 1_S(y) 1_S(z)\Big] + \frac34 \alpha^3\delta^2.
\end{align*}
The first inequality followed because $\xi^{(ij)'}((0,1,2)) \ge \frac{\eta'}{10\sqrt{n}}$ by \cref{item:lower12} of \cref{obs:joint}.
\end{proof}

\subsection{Harsh Random Restrictions}
\label{subsec:harsh}

Our goal will now be to take a random restriction down to a distribution supported on $(0,0,0)$, $(1,1,1)$, $(2,2,2)$, $(0,1,2)$ whose mass on all atoms is $\Omega(1)$. Towards this, we first prove that taking such random restrictions preserves product pseudorandomness. In fact, if the restriction is sufficiently harsh, the function becomes product pseudorandom against any fixed distribution with $\Omega(1)$ mass on each with high probability.

First, we record the simple observation that if $f$ is $1$-bounded and correlates to a $1$-bounded product function $P$, then $f$ correlates with a $1$-bounded product function with noticeable probability under random restriction.
\begin{lemma}
\label{lemma:rrprod}
Let $f: \A^n \to \bbC$ be a $1$-bounded function, and $P$ is a $1$-bounded product function. Let $\mu = \beta \mu' + (1-\beta)\nu$ be such that $\Big|\E_{x \sim \mu^{\otimes n}}[f(x)P(x)]\Big| \ge \eps$. Then
\[ \Pr_{I \sim_{1-\beta} [n], z \sim \nu^{\otimes I}}\left[\Big|\E_{x \sim (\mu')^{\otimes \bar{I}}}[f_{I \to z}(x)P_{I \to z}(x)] \Big| \ge \eps/2\right] \ge \eps/2. \]
\end{lemma}
\begin{proof}
Follows by a simple averaging argument, as $f, P$ are $1$-bounded.
\end{proof}

\begin{lemma}
\label{lemma:projpseudo}
Let $\A$ be a finite set and let $f: \A^n \to \bbC$ be a $1$-bounded function which is $(n^{1/4}, \gamma)$ product pseudorandom against a distribution $\mu$ with $\mu(x) \ge \alpha$ for all $x \in \A$. Let $\eta \ge 0$ and $\nu$ be a distribution on $\A$ with $|\nu(x)-\mu(x)| \le \eta/\sqrt{n}$ for all $x \in \A$. Then for every distribution $\mu'$ on $\A$ with mass at least $\beta \ge n^{-1/20}$ on every atom and $n^{-11/20} \le \delta \le \eta/\sqrt{n}$, the following holds for $\gamma' = 2(\gamma + 3\eta/\alpha)^{1/3}$:
\[ \Pr_{I \sim_{1-\delta} [n], z \sim \nu^{\otimes I}}\left[f_{I \to z} \text{ is } (n^{1/3}, \gamma') \text{ product pseudorandom against } \mu'\right] \ge 1-\gamma'. \]
\end{lemma}
\begin{proof}
Throughout this proof, all inner products are with respect to distributions $\mu^{\otimes n'}$ for some $n'$. For clarity, we denote this by $\l f, g\rr_{\mu} := \E_{x \sim \mu^{\otimes n'}}[f(x)g(x)]$.

Let $n' = n-|I|$, and note that $n' \ge n^{2/5}$ with high probability. If $f_{I \to z}$ is not $(n^{1/3}, \gamma')$ product pseudorandom with respect to $\mu'$, then for $\delta' = n^{1/3}/n'$, with probability at least $\gamma'$ over $I' \sim_{1-\delta'} ([n] \setminus I)$ and $z' \sim (\mu')^{\otimes I'}$, there is a $1$-bounded product function $P$ with $|\l f_{(I,I') \to (z,z')}, P \rr_{\mu'}| \ge \gamma'$.
We may write $\mu' = \beta \mu + (1-\beta) \xi$ where $\xi$ is some distribution. If $|\l f_{(I,I') \to (z,z')}, P \rr_{\mu'}| \ge \gamma'$ holds, then by \cref{lemma:rrprod}, over $I'' \sim_{1-\beta} ([n]\setminus (I \cup I'))$ and $z'' \sim \xi^{\otimes I''}$ the probability that \[ |\l f_{(I,I',I'') \to (z,z',z'')}, P_{I'' \to z''}\rr_{\mu}| \ge \gamma'/2 \] is at least $\gamma'/2$. Putting this together gives that
\[ \Pr_{I',I'',z',z''}\left[|\l f_{(I,I',I'') \to (z,z',z'')}, P \rr_{\mu}| \ge \gamma'/2 \text{ for 1-bounded } P \right] \ge (\gamma')^2/2. \]
If the conclusion of the theorem is false, then
\[ \Pr_{\substack{I \sim_{1-\delta} [n], z \sim \nu^{\otimes I} \\ I' \sim_{1-\delta'} ([n] \setminus I), z' \sim (\mu')^{\otimes I'} \\ I'' \sim_{1-\beta} ([n] \setminus (I \cup I')), z'' \sim \xi^{\otimes I''}}}\left[|\l f_{(I,I',I'') \to (z,z',z'')}, P \rr_{\mu}| \ge \gamma'/2 \text{ for 1-bounded } P \right] \ge (\gamma')^3/2 \ge 4(\gamma+3\eta/\alpha). \]
Consider the joint restriction $(I,I',I'') \to (z,z',z'')$. This is equivalent to taking a restriction $f_{J \to u}$ for $J \sim_{1-\delta''} [n], u \sim (\mu'')^{\otimes J}$ for $\delta'' = \delta\beta \delta' \ge n^{-3/4}$. Also $|\mu''(x) - \mu(x)| \le \eta/\sqrt{n} + \delta \le 2\eta/\sqrt{n}$ for all $x \in \A$. This implies that $\dtv(\mu^{\otimes n}, (\mu'')^{\otimes n}) \le 10\eta/\alpha$ because $\mu(x) \ge \alpha$ for all $x \in \A$. Applying this to the previous displayed equation gives
\begin{align*}
    \Pr_{J \sim_{1-\delta\beta \delta'} [n], u \sim \mu^{\otimes J}}\left[|\l f_{J \to u}, P \rr_{\mu}| \ge \gamma'/2 \text{ for 1-bounded } P \right] \ge 4(\gamma+3\eta/\alpha) - 10\eta/\alpha \ge \gamma,
\end{align*}
which contradicts that $f$ is $(n^{1/4}, \gamma)$ product pseudorandom.
\end{proof}

In the following arguments in this section and the next, we will heavily leverage product pseudorandomness with respect to several different distributions on $[3]$. To formally set this up, we define the following large class of distributions.
\begin{definition}
\label{def:q}
Let $\mc{Q}$ be the set of all distributions $\nu$ on a finite alphabet $\A$ such that $\nu(x)$ is a positive rational number with denominator at most $2^{1000}$ for all $x \in \A$.
\end{definition}
The idea is that by a union bound, \cref{lemma:projpseudo} implies that harsh random restrictions of $E_1, E_2$ will be product pseudorandomn against all $\nu \in \mc{Q}$ with probability at least $1 - O(\gamma')$.

Finally we random restrict the conclusion of \cref{lemma:error} to obtain a $3$-wise correlation against a distribution whose mass on all atoms is $\Omega(1)$. Let $\xi := \xi^{(ij)'}$ be as guaranteed in \cref{lemma:error}. 
\begin{lemma}
\label{lemma:atom}
Let $\mu$ be the uniform distribution on $[3]$, and let $\nu$ on $[3]^3$ be the distribution with mass $1/6$ on $(0,0,0)$ and $(0,1,2)$, and mass $1/3$ on $(1,1,1)$ and $(2,2,2)$. Let $S, E_1, E_2$ satisfying that $\mu(E_1) = \delta_1$, $\mu(E_2) = \delta_2$, and $\mu(S) \ge \alpha \mu(E_1 \boxtimes E_2)$. Assume that $1_{E_1} - \delta_1$ and $1_{E_2} - \delta_2$ are $(n^{1/4}, \gamma)$ product pseudorandom, and let $\wt{\tau} > 0$. If $S$ has no combinatorial line, at least one of the following holds. There is some $S' \subseteq E_1' \boxtimes E_2' \subseteq [3]^{n'}$ satisfying either

\noindent \textbf{(Case 1)}:
\begin{enumerate}
    \item $n' \ge n^{2/5}$, $\mu(S') \ge (\alpha + \alpha^3\wt{\tau}/1000)\mu(E_1' \boxtimes E_2')$, and $\mu(E_1' \boxtimes E_2') \ge \delta_1\delta_2/2$,
    \item If $S'$ has a combinatorial line, then $S$ does too,
\end{enumerate}
or \textbf{(Case 2)}:
\begin{enumerate}
    \item $n' \ge n^{2/5}$, $\mu(S') \ge (\alpha - \wt{\tau})\mu(E_1' \boxtimes E_2')$, and $\mu(E_1' \boxtimes E_2') \ge \delta_1\delta_2/2$,
    \item $1_{E_1'} - \delta_1$ and $1_{E_2'} - \delta_2$ are $(n^{1/3}, \gamma')$ product pseudorandom for $\gamma' = 2(\gamma + 1000\eta)^{1/3}$ with respect to all distributions $\mu' \in \mc{Q}$ supported on $\{0,1\}$ or $\{0,2\}$,
    \item $\E_{(x,y,z) \sim \nu^{\otimes n'}}\left[(1_{S'} - \alpha 1_{E_1' \boxtimes E_2'})(x) 1_{S'}(y) 1_{S'}(z) \right] \ge \alpha^3\delta_1^2\delta_2^2/4.$
    \item If $S'$ has a combinatorial line, then $S$ does too.
\end{enumerate}
\end{lemma}
\begin{proof}
Let $\beta = 6\xi((0,1,2)) \le 20\eta/\sqrt{n}$, and let $\xi = \beta \nu + (1-\beta)\nu'$ where $\nu'$ is supported on $(0,0,0)$, $(1,1,1)$, $(2,2,2)$. Note that $|\nu'(x)-1/3| \le 100\eta/\sqrt{n}$ for all $x \in [3]$. Let $I \sim_{1-\beta} [n]$, $z \sim (\nu')^{\otimes I}$, and let $S' = S_{I \to z}$, $E_1' = (E_1)_{I \to z}$, and $E_2' = (E_2)_{I \to z}$. Note that $n-|I| \ge n^{2/5}$ with high probability. We first handle the case when $\mu(S')$ drops significantly under random restriction with nonnegligible probability.

Note that with probability at least $1 - O(\gamma')$ it holds that $1_{E_1'} - \delta_1$ and $1_{E_2'} - \delta_2$ are $(n^{1/3}, \gamma')$ product pseudorandom, by \cref{lemma:projpseudo}. In one case,
\begin{align} \Pr\left[\mu(S') \le (\alpha-\wt{\tau}/2)\delta_1\delta_2 \right] \ge \alpha^3/100. \label{eq:lowprob}
\end{align}
We combine this with (using \cref{lemma:me1e2})
\begin{align*} \E[\mu(S')] &= \xi_1(S) \ge \mu(S) - \dtv(\mu^{\otimes n}, \xi_1^{\otimes n}) \\
&\ge \alpha\mu(E_1 \boxtimes E_2) - 200\eta \ge \alpha\delta_1\delta_2 - O((\gamma')^c) - 200\eta,
\end{align*}
to get that there is some $S'$ with
\[ \mu(S') \ge (\alpha + \alpha^3\wt{\tau}/500)\delta_1\delta_2 - O((\gamma')^c) - 200\eta \ge (\alpha + \alpha^3\wt{\tau}/1000)\mu(E_1' \boxtimes E_2'), \] as desired, where we have applied \cref{lemma:me1e2} again.

We now consider the other case (when \eqref{eq:lowprob} does not hold). In this case let $\mc{E}$ be the event that $1_{E_1'} - \delta_1$ and $1_{E_2'} - \delta_2$ are $(n^{1/3}, \gamma')$ product pseudorandom against all distributions in $\mc{Q}$. By \cref{lemma:projpseudo}, $\Pr[\mc{E}] \ge 1-O(\gamma')$. When $\mc{E}$ holds, \cref{lemma:eyz} proven below gives
\begin{align}
\left|\E_{(x,y,z) \sim \nu^{\otimes \bar{I}}}\left[(1_{S'} - \alpha 1_{E_1' \boxtimes E_2'})(x) 1_{S'}(y) 1_{S'}(z) \right] \right| \le \E_{(x,y,z) \sim \nu^{\otimes \bar{I}}}\left[1_{E_1' \boxtimes E_2'}(y) 1_{E_1' \boxtimes E_2'}(z) \right] \le 2\delta_1^2\delta_2^2. \label{eq:eyzyz}
\end{align}
Because $\Pr[\mc{E}] \ge 1-O(\gamma')$, applying \cref{lemma:error,lemma:me1e2} gives
\[ \E_{\substack{S', E_1', E_2' \\ \mc{E} \text{ holds}}} \E_{(x,y,z) \sim \nu^{\otimes n'}}\left[(1_{S'} - \alpha 1_{E_1' \boxtimes E_2'})(x) 1_{S'}(y) 1_{S'}(z) \right] \ge \alpha^3\delta_1^2\delta_2^2/2 - O(\gamma'). \]
Combining the previous equation with \eqref{eq:eyzyz} and an averaging argument (over $\mc{E}$) gives
\begin{align} \Pr_{S',E_1',E_2'}\left[\E_{(x,y,z) \sim \nu^{\otimes n'}}\left[(1_{S'} - \alpha 1_{E_1' \boxtimes E_2'})(x) 1_{S'}(y) 1_{S'}(z) \right] \ge \alpha^3\delta_1^2\delta_2^2/4 \right] \ge \alpha^3/20.
\label{eq:want1}
\end{align}
We turn to verifying the conditions of the lemma. By a union bound, $\mu(S') \ge (\alpha-\wt{\tau}/2)\delta_1\delta_2$, \eqref{eq:want1}, and $\mc{E}$ all hold for some $S', E_1', E_2'$. When this holds, $\mu(S') \ge (\alpha-\wt{\tau}/2)\delta_1\delta_2 \ge (\alpha-\wt{\tau})\mu(E_1' \boxtimes E_2')$ and $\mu(E_1' \boxtimes E_2') \ge \delta_1\delta_2/2$, verifying condition 1. Condition 2 holds because $E_1'$ and $E_2'$ are chosen so that $\mc{E}$ holds. Condition 3 holds because \eqref{eq:want1} holds. Condition 4 holds because $S' = S_{I \to z}$.
\end{proof}
We now prove \cref{lemma:eyz}.
\begin{lemma}
\label{lemma:eyz}
If $E_1 \subseteq \{0,1\}^n$ and $E_2 \subseteq \{0,2\}^n$ satisfy that $1_{E_1}-\delta_1$ and $1_{E_2} - \delta_2$ are $(n^{0.99}, \gamma')$ product pseudorandom with respect to $\nu$, then
\[ \E_{(x,y,z) \sim \nu^{\otimes n}}\left[1_{E_1 \boxtimes E_2}(y) 1_{E_1 \boxtimes E_2}(z) \right] \le \delta_1^2\delta_2^2 + O(\log(1/\gamma')^{-c}). \]
\end{lemma}
\begin{proof}
We can rewrite
\[ \E_{(x,y,z) \sim \nu^{\otimes n}}\left[1_{E_1 \boxtimes E_2}(y) 1_{E_1 \boxtimes E_2}(z) \right] = \E_{(x,y,z) \sim \nu^{\otimes n}}\left[1_{E_1}(y)1_{E_2}(y)1_{E_1}(z)1_{E_2}(z) \right]. \]
The support of the distribution $(\pi_1(y), \pi_2(y), \pi_1(z), \pi_2(z))$ is
\[ \{ (0,0,0,0), (1,0,1,0), (0,2,0,2), (1,0,0,2) \}. \]
It can be checked that the projections to the coordinates $234$ and $123$ are connected. Thus, the result follows expanding $1_{E_1} = \delta_1 + (1_{E_1} - \delta_1)$, $1_{E_2} = \delta_2 + (1_{E_2} - \delta_2)$, and applying \cref{thm:4csp}.
\end{proof}

\section{Density Increment}
\label{sec:ineq}

\subsection{Obtaining a Large ``Box Norm''}
\label{subsec:boxnorm}

In \cref{lemma:atom} we proved that if $S$ has no combinatorial lines, then after random restriction, the function $1_S - \alpha 1_{E_1 \boxtimes E_2}$ has a nontrivial $3$-wise correlation over the $\dhj[3]$ distribution. In this section, we will use this to obtain a density increment. To start we apply a sequence of Cauchy-Schwarz manipulations to prove the following.
\begin{theorem}
\label{thm:incr}
Let $\mu$ be a distribution supported on $(0,0,0), (1,1,1), (2,2,2), (0,1,2)$ whose mass on each atom is a rational number with denominator at most $1000$, so that $\mu_x \in \mc{Q}$. Let $f, g, h: [3]^n \to \R$ be $1$-bounded functions supported on $E_1 \boxtimes E_2$, and $\delta_1$, $\delta_2 > 0$ such that $1_{E_1}-\delta_1, 1_{E_2}-\delta_2$ are $(\gamma'n, \gamma')$ product pseudorandom with respect to every measure $\nu \in \mc{Q}$. Also,
\[ \left|\E_{(x,y,z) \sim \mu^{\otimes n}}\Big[f(x)g(y)h(z)\Big] \right| \ge \alpha^3\delta_1^2\delta_2^2/4. \]
Then there is a distribution $\mu_2$ supported on $[3]^4$ (see \cref{tab:tablemu2}) such that the marginal of $\mu_2$ onto any variable is $\mu_x$ and
\begin{align}
\E_{(x,x',x'',x''') \sim \mu_2^{\otimes n}}\Big[f(x)f(x')f(x'')f(x''') \Big] \ge \alpha^{12}\delta_1^2\delta_2^2/2^{10}. \label{eq:fakebox}
\end{align}
\end{theorem}
\cref{tab:tablemu1,tab:tablemu2,tab:tablemu3} contain distributions that we encounter in the proof.
\begin{table}[!ht]
\begin{center}
    \begin{tabular}{|c|c|c|c|c|}
    \hline
    $x$ & $y$ & $x'$ & $y'$ & $z$ \\
    \hline $0$ & $0$ & $0$ & $0$ & $0$ \\
    \hline $1$ & $1$ & $1$ & $1$ & $1$ \\
    \hline $2$ & $2$ & $2$ & $2$ & $2$ \\
    \hline $0$ & $1$ & $0$ & $1$ & $2$ \\
    \hline $2$ & $2$ & $0$ & $1$ & $2$ \\
    \hline $0$ & $1$ & $2$ & $2$ & $2$ \\
    \hline
    \end{tabular}
\end{center}
    \caption{The support of the distribution $\mu_1$}
    \label{tab:tablemu1}
\end{table}

\begin{table}[!ht]
\begin{center}
    \begin{tabular}{|c|c|c|c|c|c|c|c|}
    \hline
    $x$ & $x'$ & $x''$ & $x'''$ & $y$ & $y'$ & $z$ & $z'$ \\
    \hline $0$ & $0$ & $0$ & $0$ & $0$ & $0$ & $0$ & $0$ \\
    \hline $1$ & $1$ & $1$ & $1$ & $1$ & $1$ & $1$ & $1$ \\
    \hline $2$ & $2$ & $2$ & $2$ & $2$ & $2$ & $2$ & $2$ \\
    \hline $0$ & $0$ & $0$ & $0$ & $1$ & $1$ & $2$ & $2$ \\
    \hline $0$ & $2$ & $0$ & $2$ & $1$ & $2$ & $2$ & $2$ \\
    \hline $2$ & $0$ & $2$ & $0$ & $2$ & $1$ & $2$ & $2$ \\
    \hline $1$ & $1$ & $0$ & $0$ & $1$ & $1$ & $1$ & $2$ \\
    \hline $0$ & $0$ & $1$ & $1$ & $1$ & $1$ & $2$ & $1$ \\
    \hline
    \end{tabular}
\end{center}
    \caption{The support of the distribution $\mu_2$}
    \label{tab:tablemu2}
\end{table}

\begin{table}[!ht]
\begin{center}
    \begin{tabular}{|c|c|c|c|c|c|c|c|}
    \hline
    $\pi_1(y)$ & $\pi_1(y')$ & $\pi_1(y'')$ & $\pi_1(y''')$ & $\pi_2(z)$ & $\pi_2(z')$ & $\pi_2(z'')$ & $\pi_2(z''')$ \\
    \hline $0$ & $0$ & $0$ & $0$ & $0$ & $0$ & $0$ & $0$ \\
    \hline $1$ & $1$ & $1$ & $1$ & $0$ & $0$ & $0$ & $0$ \\
    \hline $0$ & $0$ & $0$ & $0$ & $2$ & $2$ & $2$ & $2$ \\
    \hline $1$ & $1$ & $1$ & $1$ & $2$ & $2$ & $2$ & $2$ \\
    \hline $1$ & $0$ & $1$ & $0$ & $2$ & $2$ & $2$ & $2$ \\
    \hline $0$ & $1$ & $0$ & $1$ & $2$ & $2$ & $2$ & $2$ \\
    \hline $1$ & $1$ & $1$ & $1$ & $0$ & $2$ & $0$ & $2$ \\
    \hline $1$ & $1$ & $1$ & $1$ & $2$ & $0$ & $2$ & $0$ \\
    \hline $0$ & $0$ & $1$ & $1$ & $0$ & $0$ & $2$ & $2$ \\
    \hline $1$ & $1$ & $0$ & $0$ & $2$ & $2$ & $0$ & $0$ \\
    \hline
    \end{tabular}
\end{center}
    \caption{The support of the distribution $\mu_3$}
    \label{tab:tablemu3}
\end{table}

\begin{proof}
Throughout this proof we write $E_1(x)$ and $E_2(x)$ in place of $1_{E_1}(x)$ and $1_{E_2}(x)$.
Applying the Cauchy-Schwarz inequality gives
\begin{align*}
    \alpha^6\delta_1^4\delta_2^4/16 &\le \Big(\E_{z \sim \mu_z^{\otimes n}} h(z)E_1(z)E_2(z) \E_{x,y | z} f(x)g(y) E_1(z)E_2(z) \Big)^2 \\
    \\ &\le \Big(\E_{z \sim \mu_z^{\otimes n}} E_1(z) E_2(z) \Big) \Big( \E_{(z,x,y,x',y') \sim \mu_1^{\otimes n}} f(x)f(x')g(y)g(y') E_1(z)E_2(z) \Big) \\
    &\le 1.1\delta_1\delta_2 \Big( \E_{(z,x,y,x',y') \sim \mu_1^{\otimes n}} f(x)f(x')g(y)g(y') E_1(z)E_2(z) \Big),
\end{align*}
by \cref{lemma:me1e2}. Square again and apply Cauchy-Schwarz to get
\begin{align*}
    \alpha^{12}\delta_1^8\delta_2^8/256 &\le 1.3\delta_1^2\delta_2^2 \Big(\E_{y,y' \sim (\mu_1)_{yy'}^{\otimes n}} g(y)g(y')E_1(y)E_2(y)E_1(y')E_2(y') \E_{x,x',z|y,y'}  f(x)f(x') \prod_{u\in\{y,y',z\}} E_1(u)E_2(u) \Big)^2 \\
    &\le 1.3\delta_1^2 \delta_2^2 \Big(\E_{y,y' \sim (\mu_1)_{yy'}^{\otimes n}} E_1(y)E_2(y)E_1(y')E_2(y') \Big) \\
    &~\left(\E_{(x,x',x'',x''',y,y',z,z') \sim \mu_2^{\otimes n}} f(x)f(x')f(x'')f(x''')\prod_{u\in\{x,x',x'',x''',y,y',z,z'\}} E_1(u)E_2(u) \right).
\end{align*}
By \cref{lemma:4cspy} below (whose proof uses product pseudorandomness of $E_1$ and $E_2$), the first term is bounded by $1.1\delta_1^2\delta_2^2$. Also, by examining \cref{tab:tablemu2} we can check that
\[ \prod_{u\in\{x,x',x'',x''',y,y',z,z'\}} E_1(u)E_2(u) = E_1(x)E_2(x)E_2(x')E_1(x'')E_1(y)E_1(y')E_2(z)E_2(z'). \]
Now define
\[ \omega(x,x',x'',x''') := \E_{y,y',z,z' | x,x',x'',x'''} E_1(y)E_1(y')E_2(z)E_2(z'), \]
where the expectation is over $(y,y',z,z')$ from $\mu_2$ conditioned on $x,x',x'',x'''$. By \cref{lemma:omega} below, we can write:
\begin{align*}
    \alpha^{12}\delta_1^4\delta_2^4/2^9 &\le \E_{(x,x',x'',x''') \sim \mu_2^{\otimes n}} f(x)f(x')f(x'')f(x''') \omega(x,x',x'',x''') \\
    &\le \delta_1^2\delta_2^2 \E_{(x,x',x'',x''') \sim \mu_2^{\otimes n}}\Big[f(x)f(x')f(x'')f(x''')\Big] + \E_{(x,x',x'',x''') \sim \mu_2^{\otimes n}}\Big[|\omega(x,x',x'',x''') - \delta_1^2\delta_2^2|\Big] \\
    &\le \delta_1^2\delta_2^2 \E_{(x,x',x'',x''') \sim \mu_2^{\otimes n}}\Big[f(x)f(x')f(x'')f(x''')\Big] + O((\log(1/\gamma'))^{-c}).
\end{align*}
Rearranging this implies the claim in the lemma.
\end{proof}
Now we establish the two lemmas that were required for \cref{thm:incr}.
\begin{lemma}
\label{lemma:4cspy}
It holds that
\[ \E_{y,y' \sim (\mu_1)_{yy'}^{\otimes n}} E_1(y)E_2(y)E_1(y')E_2(y') \le \delta_1^2\delta_2^2 + O((\log(1/\gamma'))^{-c}). \]
\end{lemma}
\begin{proof}
Write
\begin{align*}
\E_{y,y' \sim (\mu_1)_{yy'}^{\otimes n}} E_1(y)E_2(y)E_1(y')E_2(y') &= \E_{y,y' \sim (\mu_1)_{yy'}^{\otimes n}} \delta_1 E_2(y)E_1(y')E_2(y') \\
&+ \E_{y,y' \sim (\mu_1)_{yy'}^{\otimes n}} (E_1-\delta_1)(y)E_2(y)E_1(y')E_2(y').
\end{align*}
We claim that the term on the second line is bounded by $O((\log(1/\gamma'))^{-c})$. This follows from \cref{thm:4csp} once we verify that the corresponding $4$-ary distribution \[ (\pi_1(y), \pi_2(y), \pi_1(y'), \pi_2(y')) \in \{0, 1\} \times \{0, 2\} \times \{0, 1\} \times \{0, 2\} \] for $(y, y') \sim (\mu_1)_{yy'}$ is pairwise-connected and has connected marginals onto any $3$ of the coordinates, because $E_1 - \delta_1$ is product pseudorandom. By examining \cref{tab:tablemu1}, $(\pi_1(y), \pi_2(y), \pi_1(y'), \pi_2(y')) \in \{0,1\} \times \{0,2\} \times \{0,1\} \times \{0,2\}$ has $\Omega(1)$ mass on each of
\[ \{ (0,0,0,0), (1,0,1,0), (0,2,0,2), (0,2,1,0), (1, 0, 0, 2)) \}.\]
It can be checked that the projection of the distribution to any $3$ coordinates is connected. Now the lemma follows from repeating the decomposition three more times.
\end{proof}

The following lemma proves that $\omega(x,x',x'',x''')$ concentrates around $\delta_1^2\delta_2^2$.
\begin{lemma}
\label{lemma:omega}
It holds that
\[ \E_{(x,x',x'',x''') \sim \mu_2^{\otimes n}}\left[|\omega(x,x',x'',x''') - \delta_1^2\delta_2^2| \right] \le O((\log(1/\gamma'))^{-c}). \]
\end{lemma}
\begin{proof}
By Cauchy-Schwarz, it suffices to bound
\begin{align*}
    &\E_{(x,x',x'',x''') \sim \mu_2^{\otimes n}}\left[|\omega(x,x',x'',x''') - \delta_1^2\delta_2^2|^2 \right] \\
    = ~&\E_{(x,x',x'',x''') \sim \mu_2^{\otimes n}}\left[\Big|\E_{y,y',z,z' | x,x',x'',x'''} E_1(y)E_1(y')E_2(z)E_2(z') - \delta_1^2\delta_2^2\Big|^2 \right] \\
    = ~& \E_{(x,x',x'',x''') \sim \mu_2^{\otimes n}}\left[\Bigg|\sum_{\substack{\bE_1, \bbE_1 \in \{\delta_1,E_1-\delta_1\} \\ \bE_2, \bbE_2 \in \{\delta_2, E_2-\delta_2\} \\ (\bE_1,\bbE_1,\bE_2,\bbE_2) \neq (\delta_1,\delta_1,\delta_2,\delta_2)}} \E_{y,y',z,z' | x,x',x'',x'''} \bE_1(y)\bbE_1(y')\bE_2(z)\bbE_2(z') \Bigg|^2 \right] \\
    \le ~& 15 \max \E_{x,x',x'',x'''} \E_{\substack{y,y',y'',y''' \\ z,z',z'',z'''} | x,x',x'',x'''} \bE_1(y)\bE_1(y'')\bbE_1(y')\bbE_1(y''')\bE_2(z)\bE_2(z'')\bbE_2(z')\bbE_2(z''') \\
    = ~& 15 \max \E_{\substack{y,y',y'',y''' \\ z,z',z'',z'''} \sim \mu_3^{\otimes n}} \bE_1(y)\bE_1(y'')\bbE_1(y')\bbE_1(y''')\bE_2(z)\bE_2(z'')\bbE_2(z')\bbE_2(z'''),
\end{align*}
where $\mu_3$ is the distribution in \cref{tab:tablemu3}. In the previous expression we mean that $y$, $y'$, $y''$, $y''' \in \{0, 1\}$ and $z, z', z'', z''' \in \{0, 2\}$ because the $y$'s and $z$'s are in $\bar{E_1}$ and $\bar{E_2}$ terms respectively.
Also, the $\max$ is over all $\bE_1, \bbE_1 \in \{\delta_1, E_1 - \delta_1\}$ and $\bE_2, \bbE_2 \in \{\delta_2, E_2 - \delta_2 \}$, except for $\bE_1 = \bbE_1 = \delta_1$ and $\bE_2 = \bbE_2 = \delta_2$. Thus, the first inequality in the above display follows by expanding $E_1 = \delta_1 + (E_1 - \delta_1)$ and $E_2 = \delta_2 + (E_2 - \delta_2)$ and applying Cauchy-Schwarz. We consider the case where $\bE_1 = E_1 - \delta_1$ -- the remaining cases are similar.
\begin{align}
    &\E_{\substack{y,y',y'',y''' \\ z,z',z'',z'''} \sim \mu_3^{\otimes n}} \bE_1(y)\bE_1(y'')\bbE_1(y')\bbE_1(y''')\bE_2(z)\bE_2(z'')\bbE_2(z')\bbE_2(z''') \nonumber \\
    \le ~& \E_{\substack{y',y''' \\ z,z',z'',z'''} \sim \mu_3^{\otimes n}} \left|\E_{y,y''|y',y''',z,z',z'',z'''} \bE_1(y)\bE_1(y'') \right| \nonumber \\
    \le ~& \left(\E_{\substack{y',y''' \\ z,z',z'',z'''} \sim \mu_3^{\otimes n}} \left|\E_{y,y''|y',y''',z,z',z'',z'''} \bE_1(y)\bE_1(y'') \right|^2\right)^{1/2} \nonumber \\
    = ~& \left(\E_{\substack{y',y''' \\ z,z',z'',z'''} \sim \mu_3^{\otimes n}} \E_{\substack{y,y'' \\ \wt{y}, \wt{y}''} |y',y''',z,z',z'',z'''} \bE_1(y)\bE_1(y'')\bE_1(\wt{y})\bE_1(\wt{y}'') \right)^{1/2} \nonumber \\
    = ~& \left(\E_{(y,y'',\wt{y}, \wt{y}'') \sim \mu_4^{\otimes n}}  \bE_1(y)\bE_1(y'')\bE_1(\wt{y})\bE_1(\wt{y}'') \right)^{1/2}, \label{eq:finalomega}
\end{align}
for a distribution $\mu_4$ over $4$-tuples $(y,y'',\wt{y},\wt{y}'')$ that has mass $\Omega(1)$ on:
\[ \{ (0,0,0,0), (1,1,1,1), (0,1,0,1), (1,0,1,0), (0,0,1,1), (1,1,0,0) \}, \]
which can be seen by inspecting \cref{tab:tablemu3}. Specifically, $(0,0,0,0)$, $(1,1,1,1)$, $(0,1,0,1)$, and $(1,0,1,0)$ come from using rows $1, 2, 9, 10$ of \cref{tab:tablemu3} respectively, and $(0,0,1,1)$, $(1,1,0,0)$ come from combining rows $4$ and $6$. It can be checked that the projection of the distribution to any $3$ coordinates is connected and hence satisfies the hypotheses of \cref{thm:4csp}, and thus the quantity in \eqref{eq:finalomega} is bounded by $O((\log(1/\gamma'))^{-c})$ because $\bE_1$ is product pseudorandom. One can check that the marginals of $\mu_2$ are all $\mu_x$ due to how the sequence of Cauchy-Schwarz manipulations were applied. This completes the proof.
\end{proof}

\subsection{Random Restrictions for a Density Increment}
\label{subsec:rrdensity}

In this subsection we will take random restrictions of \eqref{eq:fakebox} to obtain a density increment, for $f := S' - \alpha E_1' \boxtimes E_2'$. Note that $f$ is nonzero on only about a $\mu(E_1' \boxtimes E_2') \approx \delta_1\delta_2$ fraction of $x$. Below we argue that when we restrict to these $x$ that the measures of $S'(x''')$ and $E_1' \boxtimes E_2'(x''')$ are not affected much.
\begin{lemma}
\label{lemma:smallmeasure}
Let $\mu$ be the marginal of $\mu_2$ onto any of $x, x', x'', x'''$. Consider $S \subseteq E_1 \boxtimes E_2 \subseteq [3]^n$ such that $\mu(E_1) = \delta_1$, $\mu(E_2) = \delta_2$, such that $E_1 - \delta_1$ and $E_2 - \delta_2$ are $(\gamma' n, \gamma')$ product pseudorandom. Then
\[ \left|\E_{(x,x''') \sim \mu_2^{\otimes n}}\left[E_1(x)E_2(x)S(x''') \right] - \delta_1\delta_2\mu(S) \right| \le O(\log(1/\gamma')^{-c}), \]
and
\[ \left|\E_{(x,x''') \sim \mu_2^{\otimes n}}\left[E_1(x)E_2(x)E_1(x''')E_2(x''') \right] - \delta_1^2\delta_2^2 \right| \le O(\log(1/\gamma')^{-c}). \]
\end{lemma}
\begin{proof}
By inspecting $\mu_2$, the corresponding distribution $(\pi_1(x), \pi_2(x), x''')$ is supported on
\[ \{(0,0,0), (1,0,1), (0,2,2), (0,0,2), (0,2,0), (1,0,0), (0,0,1) \}, \] which can easily be seen to be pairwise-connected. Thus the first result follows from \cref{thm:3csp}, after expanding $E_1 = \delta_1 + (E_1 - \delta_1)$ and $E_2 = \delta_2 + (E_2 - \delta_2)$.

For the second result, the corresponding distribution $(\pi_1(x), \pi_2(x), \pi_1(x'''), \pi_2(x'''))$ has support
\[ \{(0,0,0,0), (1,0,1,0), (0,2,0,2), (0,0,0,2), (0,2,0,0), (1,0,0,0), (0,0,1,0) \}, \]
which is again easily seen to be connected when restricted to any $3$ coordinates. Thus the result follows from \cref{thm:4csp} after expanding $E_1 = \delta_1 + (E_1 - \delta_1)$ and $E_2 = \delta_2 + (E_2 - \delta_2)$.
\end{proof}

Finally we show how to obtain a density increment if $S$ has no combinatorial lines.
\begin{theorem}
\label{thm:finalincr}
Let $\mu$ denote the uniform distribution on $[3]$. Let $S \subseteq E_1 \boxtimes E_2 \subseteq [3]^n$ such that $\mu(E_1) = \delta_1$, $\mu(E_2) = \delta_2$, and $\mu(S) = \alpha \mu(E_1 \boxtimes E_2)$. Assume that $1_{E_1} - \delta_1$ and $1_{E_2} - \delta_2$ are $(n^{1/4}, \gamma)$ product pseudorandom, and that $S$ has no combinatorial lines. Then there is some $S' \subseteq E_1' \boxtimes E_2' \subseteq [3]^{n'}$ satisfying:
\begin{enumerate}
    \item $n' \ge n^{1/3}$,
    \item $\mu(S') \ge (\alpha + \Omega(\alpha^{15}))\mu(E_1' \boxtimes E_2')$ and $\mu(E_1' \boxtimes E_2') \ge \Omega(\alpha^{12}\delta_1\delta_2)$,
    \item If $S'$ has a combinatorial line, then $S$ does too.
\end{enumerate}
\end{theorem}
\begin{proof}
Set $\wt{\tau} = c\alpha^{12}$ for sufficiently small constant $c$. Now we apply \cref{lemma:atom}. If Case 1 of \cref{lemma:atom} holds, then we are done by the choice of $\wt{\tau}$. Otherwise, Case 2 holds. In this case we may apply \cref{thm:incr} to get
\[ \E_{(x,x',x'',x''') \sim (\mu_2)^{\otimes n'}}\Big[f(x)f(x')f(x'')f(x''') \Big] \ge \alpha^{12}\delta_1^2\delta_2^2/2^{10}, \]
for $f = 1_{S'} - \alpha 1_{E_1' \boxtimes E_2'}$ for some $\mu(S') \ge (\alpha - \wt{\tau})\mu(E_1' \boxtimes E_2')$. Let $\nu$ be uniform over $(0,0,0,0)$, $(0,2,0,2)$, $(0,0,1,1)$ and write $\mu_2 = \beta\nu + (1-\beta)\nu'$ for some $\beta = \Omega(1)$ and a distribution $\nu'$. Note that $\nu$ is equivalent to $(0,\pi_2(y), \pi_1(y), y)$ for $y \sim [3]$.

Taking this random restriction gives
\[ \E_{\substack{I \sim_{1-\beta} [n'] \\ (z,z',z'',z''') \sim (\nu')^{\otimes I}}} \E_{y \sim [3]^{([n'] \setminus I)}}\Big[f_{I \to z}(0)f_{I \to z'}(\pi_2(y))f_{I \to z''}(\pi_1(y))f_{I \to z'''}(y) \Big] \ge \alpha^{12}\delta_1^2\delta_2^2/2^{10}. \]
Let $\mc{E}$ be the event that $f_{I \to z}(0) \neq 0$, i.e., that $\pi_1(z,0) \in E_1'$ and $\pi_2(z,0) \in E_2'$, where by $(z, 0) \in [3]^{n'}$ we mean that $z$ is assigned to the coordinates in $I$, and $0$ is assigned to the coordinates in $\bar{I}$.

Recall that \cref{lemma:me1e2} gives us $\left|\mu(E_1' \boxtimes E_2') - \delta_1\delta_2\right| \le O((\gamma')^c)$. Thus, conditioning on $\mc{E}$ and using $|f_{I \to z}(0)| \le 1$ gives
\begin{align}
    \E_{\substack{I \sim_{1-\beta} [n'] \\ (z,z',z'',z''') \sim (\nu')^{\otimes I} \\ \mc{E} \text{ holds}}} \left|\E_{y \sim [3]^{([n'] \setminus I)}}\Big[f_{I \to z'}(\pi_2(y))f_{I \to z''}(\pi_1(y))f_{I \to z'''}(y) \Big] \right| \ge \alpha^{12}\delta_1\delta_2/2^{11}. \label{eq:abseq}
\end{align}
For the following sequence of equations, the first equality is by the definition of conditioned on $\mc{E}$, the second equality is because $\mu$ is the uniform measure, the third equality is by properties of random restriction, and the final inequalities apply \cref{lemma:smallmeasure} and $|\Pr[\mc{E}] - \delta_1\delta_2| \le O(\gamma^c)$:
\begin{align*}
    &\E_{\substack{I \sim_{1-\beta} [n'] \\ (z,z',z'',z''') \sim (\nu')^{\otimes I} \\ \mc{E} \text{ holds}}} \left[\mu(S'_{I \to z'''}) - (\alpha - \wt{\tau} + \alpha^{12}/2^{14})\mu((E_1')_{I \to z'''} \boxtimes (E_2')_{I \to z'''}) \right] \\
    = ~& \frac{1}{\Pr[\mc{E}]} \E_{\substack{I \sim_{1-\beta} [n'] \\ (z,z',z'',z''') \sim (\nu')^{\otimes I}}} \Bigg[(E_1')_{I \to z}(0)(E_2')_{I \to z}(0)\Big(\mu(S'_{I \to z'''}) - (\alpha - \wt{\tau} + \alpha^{12}/2^{14})\mu((E_1')_{I \to z'''} \boxtimes (E_2')_{I \to z'''})\Big) \Bigg]
    \\ = ~& \frac{1}{\Pr[\mc{E}]} \E_{\substack{I \sim_{1-\beta} [n'] \\ (z,z',z'',z''') \sim (\nu')^{\otimes I}}} \E_{(y,y',y'',y''') \sim \nu^{\otimes \bar{I}}} \Bigg[(E_1')_{I \to z}(0)(E_2')_{I \to z}(0)S'_{I \to z'''}(y''') \\
    ~& - (\alpha - \wt{\tau} + \alpha^{12}/2^{14}) (E_1')_{I \to z}(0)(E_2')_{I \to z}(0)(E_1')_{I \to z'''}(y''') (E_2')_{I \to z'''}(y''')\Big) \Bigg]
    \\ = ~& \frac{1}{\Pr[\mc{E}]} \E_{(x,x''') \sim \mu_2^{\otimes n}}\left[E_1'(x)E_2'(x)S'(x''') - (\alpha-\wt{\tau}+\alpha^{12}/2^{14})E_1'(x)E_2'(x)E_1'(x''')E_2'(x''') \right] \\
    \ge ~& -\alpha^{12}\delta_1\delta_2/2^{14} - O((\log(1/\gamma')^{-c})) \ge -\alpha^{12}\delta_1\delta_2/2^{13}.
\end{align*}
Looking at the expectation on the left hand side of~\eqref{eq:abseq}, we observe that the first two
terms are analogous to new disjoint products 
that may be correlated with a restriction of $f$. 
Indeed, we note that as $f$ receives two values in $E_1'\boxtimes E_2'$, the first two terms 
form a partition of $E_1'\boxtimes E_2'$ into four
disjoint products, and~\eqref{eq:abseq} asserts that
$f$ is correlated with this partition.

We now formalize this idea. Note that $f(x) \in \{-\alpha, 0, 1-\alpha \}$, with $f(x) \neq 0$ if and only if $x \in E_1' \boxtimes E_2'$. For fixed $(z',z'')$ define
\[ F_1^+ := \{y \in \{0,1\}^{[n] \setminus I} : f_{I \to z''}(y) = 1-\alpha \} \enspace \text{ and } \enspace F_1^- := \{y \in \{0,1\}^{[n] \setminus I} : f_{I \to z''}(y) = -\alpha\}. \] Define
\[ F_2^+ := \{y \in \{0,2\}^{[n] \setminus I} : f_{I \to z'}(y) = 1-\alpha \} \enspace \text{ and } \enspace F_2^- := \{ y \in \{0,2\}^{[n] \setminus I} : f_{I \to z'}(y) = -\alpha \}. \] By examining $\mu_2$ one can check that $F_1^+ \cup F_1^- = (E_1')_{I \to z'''}$ an $F_2^+ \cup F_2^- = (E_2')_{I \to z'''}$.

Now the above equation gives us:
\begin{align}
    &\E_{\substack{I \sim_{1-\beta} [n'] \\ (z,z',z'',z''') \sim (\nu')^{\otimes I} \\ \mc{E} \text{ holds}}} \E_{\substack{F_1 \in \{F_1^+, F_1^-\} \\ F_2 \in \{F_2^+, F_2^-\}}} \left[\mu(S'_{I \to z'''} \cap (F_1 \boxtimes F_2)) - (\alpha - \wt{\tau} + \alpha^{12}/2^{14})\mu(F_1 \boxtimes F_2) \right] \nonumber \\
    = ~& \frac14 \E_{\substack{I \sim_{1-\beta} [n'] \\ (z,z',z'',z''') \sim (\nu')^{\otimes I} \\ \mc{E} \text{ holds}}} \left[\mu(S'_{I \to z'''}) - (\alpha - \wt{\tau} + \alpha^{12}/2^{14})\mu((E_1')_{I \to z'''} \boxtimes (E_2')_{I \to z'''}) \right] \nonumber
    \\ \ge ~& -\alpha^{12}\delta_1\delta_2/2^{15}. \label{eq:noabseq}
\end{align}
$f_{I \to z''}$ is constant on any $F_1 \in \{F_1^+, F_1^-\}$, and we denote its value by $f_{I \to z''}(F_1)$. We similarly define $f_{I \to z'}(F_2)$. Applying the triangle inequality, and then \eqref{eq:abseq} and \cref{lemma:smallmeasure} gives
\begin{align}
    &\E_{\substack{I \sim_{1-\beta} [n'] \\ (z,z',z'',z''') \sim (\nu')^{\otimes I} \\ \mc{E} \text{ holds}}} \E_{\substack{F_1 \in \{F_1^+, F_1^-\} \\ F_2 \in \{F_2^+, F_2^-\}}} \left|\mu(S'_{I \to z'''} \cap (F_1 \boxtimes F_2)) - (\alpha - \wt{\tau} + \alpha^{12}/2^{14})\mu(F_1 \boxtimes F_2) \right|  \nonumber \\
    \ge ~& \E_{\substack{I \sim_{1-\beta} [n'] \\ (z,z',z'',z''') \sim (\nu')^{\otimes I} \\ \mc{E} \text{ holds}}} \left|\E_{\substack{F_1 \in \{F_1^+, F_1^-\} \\ F_2 \in \{F_2^+, F_2^-\}}} f_{I \to z'}(F_2)f_{I \to z''}(F_1)\left(\mu(S'_{I \to z'''} \cap (F_1 \boxtimes F_2)) - \alpha \cdot \mu(F_1 \boxtimes F_2)\right) \right| \nonumber \\
    - ~& \frac{\alpha^{12}}{2^{14}}\E_{\substack{I \sim_{1-\beta} [n'] \\ (z,z',z'',z''') \sim (\nu')^{\otimes I} \\ \mc{E} \text{ holds}}} \E_{\substack{F_1 \in \{F_1^+, F_1^-\} \\ F_2 \in \{F_2^+, F_2^-\}}} \mu(F_1 \boxtimes F_2) \nonumber \\
    = ~&
    \frac14 \E_{\substack{I \sim_{1-\beta} [n'] \\ (z,z',z'',z''') \sim (\nu')^{\otimes I} \\ \mc{E} \text{ holds}}} \left|\E_{y \sim [3]^{([n'] \setminus I)}}\Big[f_{I \to z'}(\pi_2(y))f_{I \to z''}(\pi_1(y))f_{I \to z'''}(y) \Big] \right| \nonumber \\
    - ~& \frac{1}{\Pr[\mc{E}]} \cdot \alpha^{12}/2^{16} \cdot \E_{(x,x''') \sim \mu_2^{\otimes n}}\left[E_1(x)E_2(x)E_1(x''')E_2(x''') \right] \nonumber
    \\ \ge ~& \alpha^{12}\delta_1\delta_2/2^{13} - \alpha^{12}\delta_1\delta_2/2^{15} \ge \alpha^{12}\delta_1\delta_2/2^{14}. \label{eq:abseq2}
\end{align}
Consider the random variable, defined on $I \sim_{1-\beta}[n']$, $(z,z',z'',z''') \sim (\nu')^{\otimes I}$, and 
$F_1 \in \{F_1^+, F_1^-\}$, $F_2 \in \{F_2^+, F_2^-\}$,
\[ X = \mu(S'_{I \to z'''} \cap (F_1 \boxtimes F_2)) - (\alpha - \wt{\tau} + \alpha^{12}/2^{14})\mu(F_1 \boxtimes F_2). \]
\eqref{eq:noabseq} and \eqref{eq:abseq2} together imply that \[ \E[X + |X|] \ge -\alpha^{12}\delta_1\delta_2/2^{15} + \alpha^{12}\delta_1\delta_2/2^{14} = \alpha^{12}\delta_1\delta_2/2^{15}.\] Hence there is a realization of $X \ge 0$ with $X = \frac12(X+|X|) \ge \alpha^{12}\delta_1\delta_2/2^{16}$. By the definition of $X$, there is some choice of $(z,z',z'',z''')$ and $F_1, F_2$ such that
\[ \mu(S'_{I \to z'''} \cap (F_1 \boxtimes F_2)) \ge (\alpha - \wt{\tau} + \alpha^{12}/2^{14})\mu(F_1 \boxtimes F_2) + \Omega(\alpha^{12}\delta_1\delta_2). \]
This implies the conditions because $\beta n' \gg n^{1/3}$, $\wt{\tau} = c\alpha^{12}$ for small constant $c$, and
\[ \mu(F_1 \boxtimes F_2) \ge \mu(S'_{I \to z'''} \cap (F_1 \boxtimes F_2)) \ge \Omega(\alpha^{12}\delta_1\delta_2). \]
\end{proof}

\section{Uniformization}
\label{sec:uniform}

In this section we explain how to start with a subset $S \subseteq E_1 \boxtimes E_2$ of density $\alpha$, where $E_1, E_2$ are not necessarily product pseudorandom in the sense of \cref{def:prodpseudo}, and to reach a situation where $E_1, E_2$ are, without decreasing the measure of $S$ or $n$ by too much. The proof will use two types of random restrictions: the standard one introduced in \cref{def:rr} as well as one which fixes many coordinates to take the same value (see \cref{def:same}). Note that applying such restrictions cannot create new combinatorial lines.

\begin{theorem}
\label{thm:uniform}
Let $S \subseteq E_1 \boxtimes E_2 \subseteq [3]^n$ with $\mu(E_1 \boxtimes E_2) = \delta$ and $\mu(S) \ge (\alpha+\tau)\delta$, and $0 < \gamma < (\delta\tau)^{10}$. There is some $n' \ge n^{\exp(-O(\delta^{-1}\tau^{-1}\gamma^{-4}))}$ and $S' \subseteq E_1' \boxtimes E_2' \subseteq [3]^{n'}$ such that:
\begin{enumerate}
    \item $\mu(E_1' \boxtimes E_2') \ge \delta\tau/3$.
    \item $(S', E_1', E_2') \in \struct_{\alpha+\tau/2}$.
    \item If $S'$ contains a combinatorial line, then $S$ does too.
\end{enumerate}
\end{theorem}

We mimic the proof of the uniformization step in Shkredov's bound for the corners problem \cite{Shk05,Shk06} (also see \cite{Green04,Green05} for notes from which we borrowed notation). The high level idea is that if $E_1$ (or $E_2$) is not product function pseudorandom, then we can ``partition'' $[3]^n$ into smaller sets of the form $E_1' \boxtimes E_2'$ so that the measures of these $E_1'$ have higher variance than before. This process eventually terminates, and we find a partition piece under which $S$ still has large enough density. One step of this process amounts to being able to density increment sets that correlate to product functions (after random restriction), which was exactly done in \cite[Section 8]{csp6}.

Now we move towards formally proving \cref{thm:uniform}. Let $\xi$ be some distribution over triples $(S', E_1', E_2')$ with $S' \subseteq E_1' \boxtimes E_2' \subseteq [3]^{n'}$ for some $n'$. Define the \emph{index} of $\xi$ to be:
\[ \mc{I}(\xi) := \E_{(S',E_1',E_2') \sim \xi}[\mu(E_1')^2 + \mu(E_2')^2]. \]
Note that $\mc{I}(\xi) \le 2$ always. The following lemma shows how to perform one round of partitioning.
\begin{lemma}
\label{lemma:oneround}
Let $S \subseteq E_1 \boxtimes E_2 \subseteq [3]^n$ where $1_{E_1}-\mu(E_1)$ or $1_{E_2} - \mu(E_2)$ is not $(n^{1/4}, \gamma)$ product-pseudorandom. Then there is a distribution $\xi$ over triples $(S', E_1', E_2')$ satisfying, for $c = 1/1000$:
\begin{enumerate}
    \item $S' \subseteq E_1' \boxtimes E_2' \subseteq [3]^{n'}$ for $n' \ge n^c$.
    \item $\Big|\E_{(S', E_1', E_2') \sim \xi}[\mu(S')] - \mu(S)\Big| \le n^{-c}$.
    \item $\Big|\E_{(S', E_1', E_2') \sim \xi}[\mu(E_i')] - \mu(E_i)\Big| \le n^{-c}$ for $i = 1, 2$.
    \item $\Big|\E_{(S', E_1', E_2') \sim \xi}[\mu(E_1' \boxtimes E_2')] - \mu(E_1 \boxtimes E_2)\Big| \le n^{-c}$.
    \item For all $(S', E_1', E_2') \in \supp(\xi)$, if $S'$ has a combinatorial line, then $S$ does too.
    \item (Increment) $\mc{I}(\xi) \ge (\mu(E_1)^2 + \mu(E_2)^2) + \gamma^4/2$.
\end{enumerate}
\end{lemma}
Condition $1$ simply says that the new instances we produce are in a polynomially large number of dimensions. Conditions $2$-$4$ say that the densities of the objects we care about, namely $S, E_1, E_2, E_1 \boxtimes E_2$, only change by a negligible amount. Condition $5$ says that if the new instances have combinatorial lines, then the original instance did too. Finally, Condition $6$ says that the index of the set has increased.

Beyond standard random restrictions, the proof of \cref{lemma:oneround} also requires another operation to generate the triples $(S', E_1', E_2')$ that we increment onto. This operation takes a subset $T \subseteq [n]$ and forces that $x_t = x_{t'}$ for all $t, t' \in T$.
\begin{definition}
\label{def:same}
For a function $g: [3]^n \to \bbC$ and $T \subseteq [n]$, define $g_{=T}: [3]^{n-|T|+1} \to \bbC$ as follows. For $x \in [3]$ and $y \in [3]^{[n] \setminus T}$ let $v \in [3]^n$ be the vector where $v_i = y_i$ for $i \in [n] \setminus T$ and $v_i = x$ otherwise. Then $g_{=T}(x,y) := g(v)$. For disjoint sets $T_1, \dots, T_N$ we also write $g_{=T_1,\dots,T_N} = (\dots(g_{=T_1})\dots)_{=T_N}$.
\end{definition}
Our next lemma argues that if $T$ is a random subset of $S \subseteq [n]$ of size much smaller than $\sqrt{|S|}$, then performing the operation described in \cref{def:same} does not significantly affect the TV distance on the uniform distribution. This is more or less a special case of \cite[Lemma 6.3]{DHJ12}, but we provide a proof sketch for this easier case for completeness.
\begin{lemma}
\label{lemma:sametv}
Let $\nu$ be the distribution on $[3]^n$ over the $v$ vector (as defined in \cref{def:same}) obtained by sampling $T \subseteq S, |T| = k$ and $x \in [3]^{n-k+1}$ uniformly. Then $\dtv(\nu, \mu) \le \frac{10k}{\sqrt{|S|}}$, where $\mu$ is the uniform distribution on $[3]^n$.
\end{lemma}
\begin{proof}
Further let $\nu'$ to be the distribution where we sample $T \subseteq S$ and set $x_t = a$ for all $t \in T$, where $a$ is some fixed element of $[3]$. We will prove that $\dtv(\nu', [3]^n) \le \frac{10k}{\sqrt{|S|}}$. When $k = 1$ the claim follows by a direct calculation. For $k > 1$ note that sampling $T \subseteq S$ can be achieved by sampling $T' \subseteq S$ with $|T| = k-1$ and then sampling $t \in S \setminus T'$ and setting $T = T' \cup \{t\}$. Thus the result follows by induction.
\end{proof}
We are now in a position to prove \cref{lemma:oneround}.
\begin{proof}[Proof of \cref{lemma:oneround}]
We consider the case where $1_{E_1} - \mu(E_1)$ is not product pseudorandom -- the case of $E_2$ is similar. Throughout this proof we will use $E_1(x)$ in place of $1_{E_1}(x)$ (and similar for $E_2, S$). Define the sets $(E_1)_{I \to z} \subseteq \{0, 1\}^{[n]\setminus I}$ and $(E_1)_{=T} \subseteq \{0,1\}^{n-|T|+1}$ in the expected way.

Let $f := E_1 - \mu(E_1)$. The first step is to take some $I \subseteq [n]$ with $|I| \le n-n^{1/4}/2$ such that
\[ \Pr_{z \sim \mu^{\otimes I}}\left[\exists \{P_i: [3] \to \bbC, \|P_i\|_\infty \le 1\}_{i \in \bar{I}} \enspace \text{ with } \enspace \Big|\E_{x \sim \mu^{\otimes \bar{I}}}\Big[f_{I\to z}(x) \prod_{i \in \bar{I}} P_i(x_i) \Big]\Big| \ge \gamma \right] \ge \gamma, \] which is guaranteed to exist because $f$ is not $(n^{1/4}, \gamma)$ product pseudorandom.

For $z \in [3]^I$ define $S' = S_{I \to z}$, $E_1' = (E_1)_{I \to z}$, $E_2' = (E_2)_{I \to z}$, and $m = n - |I| \ge n^{1/4}/2$. Relabel $\bar{I}$ as $[m]$. For $\gamma$ fraction of such $z$, there are functions $P_1, \dots, P_m$ such that:
\begin{align} \Big|\E_{x \sim \mu^{\otimes m}}\Big[f_{I\to z}(x) \prod_{i=1}^m P_i(x_i) \Big]\Big| \ge \gamma. \label{eq:zsat} \end{align}
We focus on such a $z$ and perform further operations to $(S', E_1', E_2')$. Let $v_1, \dots, v_m: [3] \to \R/\Z$ be such that $P_j(x) = e^{2\pi i v_j(x)}$ for all $j \in [m]$, $x \in [3]$.
Let $\zeta = \frac{1}{72}$, and let $N = m^{\zeta}$. For $x \in (\R/\Z)^n$ let $\|x\|_\infty := \max_{i\in[n]} |x_i|_{\R/\Z}$.
By the Pigeonhole principle, we can find disjoint sets $S_1, \dots, S_N$ of size $|S_i| = \sqrt{m}/2$ such that for every $i \in [N]$ and $j, j' \in S_i$, it holds that $\|v_j - v_{j'}\|_\infty \le m^{-\frac{1}{6}}$. Let $v$ be an arbitrary vector in $S_i$, and let $1 \le k_i \le m^{\frac{1}{12}}$ be such that $\|k_iv\|_\infty \le m^{-\frac{1}{36}}$ -- such a $k_i$ exists by Dirichlet approximation.

Now let $T_i$ be a uniformly random subset of $S_i$ of size $k_i$ for all $i \in [N]$ and perform the following sequence of restrictions. For all $i \in [N]$ force that $x_j = x_{j'}$ for all $j, j' \in T_i$, and then uniformly random restrict all coordinates in $J := [m] \setminus (T_1 \cup \dots \cup T_N)$. Let $P = P_1\dots P_m$ and $Q = (P_{=T_1,\dots,T_N})_{J\to u}$. We will argue that $Q$ is nearly constant. For $x, y \in [3]^N$,
\begin{align*}
    Q(x) - Q(y) = \exp\Big(2\pi i \sum_{j=1}^N \sum_{t \in T_j} v_t(y)\Big)\left(\exp\Big(2\pi i \sum_{j=1}^N \sum_{t \in T_j} (v_t(x)-v_t(y))\Big) - 1 \right)\prod_{i \in ([m] \setminus \cup_i T_i)} P_i(u_i).
\end{align*}
For $j \in [N]$ and let $v \in T_j$ be such that $\|k_jv\|_\infty \le m^{-\frac{1}{36}}$. Then
\[ \Big\|\sum_{t \in T_j} v_t \Big\|_\infty \le \|k_jv\|_\infty + \sum_{t \in T_j} \|v_t - v\|_\infty \le m^{-\frac{1}{36}} + k_j m^{-\frac{1}{6}} \le 2m^{-\frac{1}{36}}. \]
Combining this with the above yields that $|Q(x) - Q(y)| \le 100N m^{-\frac{1}{36}} = 100m^{-\frac{1}{72}}.$

By \cref{lemma:sametv} we know that
\begin{align*}
    &\E_{\substack{T_1,\dots,T_N \\ u \sim [3]^J}} \left|\E_{x \sim [3]^N}\left[(f_{I\to z, J \to u})_{=T_1,\dots,T_N}(x) Q(x) \right] \right| \\
    \ge ~&\left|\E_{x \sim [3]^m} f_{I\to z}(x) P(x) \right| - \frac{10\sum_{i=1}^N k_i}{(\sqrt{m}/2)^{1/2}}
    \ge \left|\E_{x \sim [3]^m} f_{I\to z}(x) P(x) \right| - \frac{10Nm^{1/12}}{m^{1/4}/2} \\
    \ge ~& \gamma - m^{-1/8}.
\end{align*}
By using that $Q$ is nearly constant above, the previous inequality becomes:
\begin{align}
\E_{\substack{T_1,\dots,T_N \\ u \sim [3]^J}} \left|\mu(((E_1)_{I \to z, J \to u})_{=T_1,\dots,T_N}) - \mu(E_1)\right| \ge \gamma - m^{-1/8} - 100m^{-1/72} \ge \gamma - 101m^{-1/72}. \label{eq:bias}
\end{align}
By \cref{lemma:sametv} again we also know that:
\begin{align}
\left|\E_{\substack{T_1,\dots,T_N \\ u \sim [3]^J}} \left[\mu(((E_1)_{I \to z, J \to u})_{=T_1,\dots,T_N}) \right] - \mu((E_1)_{I \to z}) \right| \le m^{-1/8}. \label{eq:bias2}
\end{align}
Finally, define distribution $\xi$ over triples $(S'', E_1'', E_2'')$ as follows. Let $z \sim [3]^I$. If $z$ did not satisfy \eqref{eq:zsat} for any product function $P = P_1 \dots P_m$ then set $(S'', E_1'', E_2'') = (S', E_1', E_2')$, where recall that $S' = S_{I \to z}$, $E_1' = (E_1)_{I \to z}$, and $E_2' = (E_2)_{I \to z}$. If $z$ satisfies \eqref{eq:zsat} then set $S'' = (S_{I \to z, J \to u})_{=T_1,\dots,=T_N}$, $E_1'' = ((E_1)_{I \to z, J \to u})_{=T_1,\dots,=T_N}$, and $E_2'' = ((E_2)_{I \to z, J \to u})_{=T_1,\dots,=T_N}$ for $T_1, \dots, T_N$ and $u$ sampled as described above. We verify the conditions $1$-$6$ of the lemma.

Condition $1$ follows by construction, as $N := m^{1/72} \ge (n^{1/4}/2)^{1/72} \ge n^{1/300}$. Conditions $2$-$4$ follow by applications of \cref{lemma:sametv} in the same way as \eqref{eq:bias2}.
Condition $5$ follows by constrction: restrictions and $=T$ operations can only reduce the number of combinatorial lines.

To see Condition $6$, we write
\begin{align*}
    \E[\mu(E_1'')^2] &= \E\left[\Big|\mu(E_1'') - \E[\mu(E_1'')]\Big|^2 \right] + \E[\mu(E_1'')]^2 \\
    &\ge \left(\gamma(\gamma - 101m^{-1/72})\right)^2 + (\mu(E_1) - m^{-1/8})^2 \ge \mu(E_1)^2 + 0.6\gamma^4,
\end{align*}
where the first inequality uses \eqref{eq:bias} and the fact that a $\gamma$ fraction of $z$'s satisfy \eqref{eq:zsat}, and \eqref{eq:bias}. Also by Jensen, $\E[\mu(E_2'')^2] \ge \E[\mu(E_2'')]^2 \ge \mu(E_2)^2 - m^{-1/8}$, where we have used Condition $3$. Condition $6$ follows by combining these.
\end{proof}

To prove \cref{thm:uniform} we repeatedly apply \cref{lemma:oneround} until the mass of the triples $(S', E_1', E_2')$ that are not product pseudorandom is small.
\begin{proof}[Proof of \cref{thm:uniform}]
Call a triple $(S', E_1', E_2')$ with $S' \subseteq E_1' \boxtimes E_2' \subseteq [3]^{n'}$ \emph{good} if both $1_{E_1'} - \mu(E_1')$ and $1_{E_2'} - \mu(E_2')$ are $((n')^{1/4}, \gamma)$ product pseudorandom.

Let $\xi^{(0)}$ have mass $1$ on $(S, E_1, E_2)$. Do the following algorithm: for $t \ge 0$ terminate if
\[ \Pr_{(S',E_1',E_2') \sim \xi^{(t)}}[(S',E_1',E_2') \text{ is not good}] \le \delta\tau/100. \]
Otherwise, define $\xi^{(t+1)}$ as follows. Sample $(S', E_1', E_2') \sim \xi^{(t)}$. If $(S', E_1', E_2')$ is good, then keep that mass in $\xi^{(t+1)}$. Otherwise, sample $(S'', E_1'', E_2'')$ from the distribution guaranteed by \cref{lemma:oneround} and put that mass in $\xi^{(t+1)}$. By \cref{lemma:oneround} we know that $\mc{I}(\xi^{(t+1)}) \ge \mc{I}(\xi^{(t)}) + \delta\tau\gamma^4/200$, and hence the algorithm terminates within $T := O(\delta^{-1}\tau^{-1}\gamma^{-4})$ steps.

Let $n' = n^{1000^{-T}}$. By applying \cref{lemma:oneround} inductively we know that
\[ \Big|\E_{(S',E_1',E_2') \sim \xi^{(T)}}[\mu(S')] - \mu(S) \Big| \le T/n' \enspace \text{ and } \enspace \Big|\E_{(S',E_1',E_2') \sim \xi^{(T)}}[\mu(E_1' \boxtimes E_2')] - \mu(E_1 \boxtimes E_2) \Big| \le T/n'. \]
Thus we get that for $c' = 1000^{-T}$,
\begin{align*}
    &\E_{(S',E_1',E_2') \sim \xi^{(T)}}\left[1_{(S',E_1',E_2') \text{ good}}\Big(\mu(S') - (\alpha+\tau/2)\mu(E_1' \boxtimes E_2') \Big) \right] \\
    \ge ~& \mu(S) - (\alpha + \tau/2)\mu(E_1 \boxtimes E_2) - 2Tn^{-c'} - \delta\tau/100 \ge \delta\tau/3,
\end{align*}
for our choice of parameters. Thus there is $(S', E_1', E_2') \in \supp(\xi^{(T)})$ that is good, $\mu(S') \ge (\alpha+\tau/2)\mu(E_1' \boxtimes E_2')$. For such choice, $\mu(E_1' \boxtimes E_2') \ge \mu(S') \ge \delta\tau/3$ as $S'\subseteq E_1'\boxtimes E_2'$. Also, $S' \subseteq [3]^{n'}$, which concludes the proof.
\end{proof}

\begin{proof}[Proof of~\cref{thm:main}]
Assume for contradiction that $S$ has no combinatorial line. Initially, for $E_1 = \{0,1\}^n$ and $E_2 = \{0,2\}^n$ and $|S| = \alpha_0 \cdot 3^n$, it holds that $(S, E_1, E_2) \in \struct_{\alpha_0}$. If $(S, E_1, E_2) \in \struct_{\alpha}$ for $S \subseteq [3]^n$ and $\delta := \mu(E_1 \boxtimes E_2)$, then applying \cref{thm:finalincr,thm:uniform} gives a new $(S', E_1', E_2') \in \struct_{\alpha + \Omega(\alpha^{15})}$ where $S' \subseteq [3]^{n'}$, and $n' \ge n^{1000^{-O(\gamma^{-4}\delta^{-1}\alpha^{-O(1)})}}$. Also if $S'$ has a combinatorial line, then $S$ does too and $\mu(E_1' \boxtimes E_2') \ge \Omega(\alpha^{O(1)})\delta$.

Thus the process terminates in $O(\alpha_0^{-15})$ steps, and $\delta = \mu(E_1 \boxtimes E_2) \ge \exp(-\alpha_0^{-O(1)})$ throughout the process. Throughout, we have required that $\log(1/\gamma)^{-c} < \delta^{O(1)}$, which holds for $\gamma = \exp(-\exp(\alpha_0^{-O(1)}))$. For $\alpha_0 = (\log\log\log\log n)^{-c}$ for sufficiently small $c$, the dimension during the iterative process is always at least
\[ n^{1000^{-O(\gamma^{-4}\delta^{-1}\alpha_0^{-O(1)})}} \ge \exp((\log n)^{0.999}). \]
This completes the proof.
\end{proof}

\section*{Acknowledgments}
We would like to thank Pratik
Worah, Madhur
Tulsiani, Gabor Kun, Per Austrin, Johan H\r{a}stad, and Ame Khot for several
discussions
that helped start the project.

{\small
\bibliographystyle{alpha}
\bibliography{refs}}

\begin{thebibliography}{BKLM24b}

\bibitem[AS74]{AS74}
M.~Ajtai and E.~Szemer\'edi.
\newblock Sets of lattice points that form no squares.
\newblock {\em Studia Sci. Math. Hungar.}, 9:9--11, 1974.

\bibitem[Aus11]{Austin11}
Tim Austin.
\newblock Deducing the density {H}ales--{J}ewett theorem from an infinitary
  removal lemma.
\newblock {\em Journal of Theoretical Probability}, 24:615--633, 2011.

\bibitem[Beh46]{B46}
F.~A. Behrend.
\newblock On sets of integers which contain no three terms in arithmetical
  progression.
\newblock {\em Proc. Nat. Acad. Sci. U.S.A.}, 32:331--332, 1946.

\bibitem[BKLM24a]{csp6}
Amey Bhangale, Subhash Khot, {Yang P.} Liu, and Dor Minzer.
\newblock On approximability of satisfiable $k$-csps: {VI}.
\newblock 2024.

\bibitem[BKLM24b]{csp7}
Amey Bhangale, Subhash Khot, {Yang P.} Liu, and Dor Minzer.
\newblock On approximability of satisfiable $k$-csps: {VII}.
\newblock 2024.

\bibitem[BKM22]{BKM1}
Amey Bhangale, Subhash Khot, and Dor Minzer.
\newblock On approximability of satisfiable \emph{k}-csps: {I}.
\newblock In Stefano Leonardi and Anupam Gupta, editors, {\em {STOC} '22: 54th
  Annual {ACM} {SIGACT} Symposium on Theory of Computing, Rome, Italy, June 20
  - 24, 2022}, pages 976--988. {ACM}, 2022.

\bibitem[BKM23a]{BKM3ap}
Amey Bhangale, Subhash Khot, and Dor Minzer.
\newblock Effective bounds for restricted $3$-arithmetic progressions in
  $\mathbb{F}_p^n$.
\newblock {\em Electron. Colloquium Comput. Complex.}, {TR23-116}, 2023.
\newblock To appear in Discrete Analysis.

\bibitem[BKM23b]{BKM2}
Amey Bhangale, Subhash Khot, and Dor Minzer.
\newblock On approximability of satisfiable $k$-csps: {II}.
\newblock In Barna Saha and Rocco~A. Servedio, editors, {\em Proceedings of the
  55th Annual {ACM} Symposium on Theory of Computing, {STOC} 2023, Orlando, FL,
  USA, June 20-23, 2023}, pages 632--642. {ACM}, 2023.

\bibitem[BKM23c]{BKM3}
Amey Bhangale, Subhash Khot, and Dor Minzer.
\newblock On approximability of satisfiable k-csps: {III}.
\newblock In Barna Saha and Rocco~A. Servedio, editors, {\em Proceedings of the
  55th Annual {ACM} Symposium on Theory of Computing, {STOC} 2023, Orlando, FL,
  USA, June 20-23, 2023}, pages 643--655. {ACM}, 2023.

\bibitem[BKM24a]{BKM4}
Amey Bhangale, Subhash Khot, and Dor Minzer.
\newblock On approximability of satisfiable k-csps: {IV}.
\newblock In Bojan Mohar, Igor Shinkar, and Ryan O'Donnell, editors, {\em
  Proceedings of the 56th Annual {ACM} Symposium on Theory of Computing, {STOC}
  2024, Vancouver, BC, Canada, June 24-28, 2024}, pages 1423--1434. {ACM},
  2024.

\bibitem[BKM24b]{BKM5}
Amey Bhangale, Subhash Khot, and Dor Minzer.
\newblock On approximability of satisfiable k-csps: {V}.
\newblock {\em CoRR}, abs/2408.15377, 2024.

\bibitem[BS23]{BS23}
Thomas~F Bloom and Olof Sisask.
\newblock An improvement to the {K}elley-{M}eka bounds on three-term arithmetic
  progressions.
\newblock {\em arXiv preprint arXiv:2309.02353}, 2023.
\newblock Available at \url{https://arxiv.org/pdf/2309.02353}.

\bibitem[DKT14]{DKT14}
Pandelis Dodos, Vassilis Kanellopoulos, and Konstantinos Tyros.
\newblock A simple proof of the density {H}ales-{J}ewett theorem.
\newblock {\em Int. Math. Res. Not. IMRN}, (12):3340--3352, 2014.

\bibitem[FK78]{FK78}
H.~Furstenberg and Y.~Katznelson.
\newblock An ergodic {S}zemer\'edi theorem for commuting transformations.
\newblock {\em J. Analyse Math.}, 34:275--291, 1978.

\bibitem[FK89]{FK89}
H.~Furstenberg and Y.~Katznelson.
\newblock A density version of the {H}ales-{J}ewett theorem for {$k=3$}.
\newblock {\em Discrete Math.}, 75(1-3):227--241, 1989.
\newblock Graph theory and combinatorics (Cambridge, 1988).

\bibitem[FK91]{FK91}
H.~Furstenberg and Y.~Katznelson.
\newblock A density version of the {H}ales-{J}ewett theorem.
\newblock {\em J. Anal. Math.}, 57:64--119, 1991.

\bibitem[Gow98]{Gowers98}
W.~T. Gowers.
\newblock A new proof of {S}zemer\'edi's theorem for arithmetic progressions of
  length four.
\newblock {\em Geom. Funct. Anal.}, 8(3):529--551, 1998.

\bibitem[Gow01]{Gowers01}
W.~T. Gowers.
\newblock A new proof of {S}zemer\'edi's theorem.
\newblock {\em Geom. Funct. Anal.}, 11(3):465--588, 2001.

\bibitem[Gre04]{Green04}
Ben Green.
\newblock Finite field models in additive combinatorics.
\newblock {\em arXiv preprint math/0409420}, 2004.
\newblock Available at \url{https://arxiv.org/pdf/math/0409420}.

\bibitem[Gre05a]{Green05}
Ben Green.
\newblock An argument of {S}hkredov in the finite field setting.
\newblock {\em Preprint}, 2005.
\newblock Available at
  \url{https://people.maths.ox.ac.uk/greenbj/papers/corners.pdf}.

\bibitem[Gre05b]{Green05prime}
Ben Green.
\newblock Roth's theorem in the primes.
\newblock {\em Ann. of Math. (2)}, 161(3):1609--1636, 2005.

\bibitem[GT08]{GT08}
Ben Green and Terence Tao.
\newblock The primes contain arbitrarily long arithmetic progressions.
\newblock {\em Ann. of Math. (2)}, 167(2):481--547, 2008.

\bibitem[GT10]{GT10}
Benjamin Green and Terence Tao.
\newblock Linear equations in primes.
\newblock {\em Ann. of Math. (2)}, 171(3):1753--1850, 2010.

\bibitem[GT12]{GT12}
Ben Green and Terence Tao.
\newblock The quantitative behaviour of polynomial orbits on nilmanifolds.
\newblock {\em Ann. of Math. (2)}, 175(2):465--540, 2012.

\bibitem[GT17]{GT17}
Ben Green and Terence Tao.
\newblock New bounds for {S}zemer\'edi's theorem, {III}: a polylogarithmic
  bound for {$r_4(N)$}.
\newblock {\em Mathematika}, 63(3):944--1040, 2017.

\bibitem[GTZ11]{GTZ11}
Ben Green, Terence Tao, and Tamar Ziegler.
\newblock An inverse theorem for the {G}owers {$U^4$}-norm.
\newblock {\em Glasg. Math. J.}, 53(1):1--50, 2011.

\bibitem[GTZ12]{GTZ12}
Ben Green, Terence Tao, and Tamar Ziegler.
\newblock An inverse theorem for the {G}owers {$U^{s+1}[N]$}-norm.
\newblock {\em Ann. of Math. (2)}, 176(2):1231--1372, 2012.

\bibitem[HJ63]{HJ63}
A.~W. Hales and R.~I. Jewett.
\newblock Regularity and positional games.
\newblock {\em Trans. Amer. Math. Soc.}, 106:222--229, 1963.

\bibitem[HLY21]{HLY20}
Rui Han, Michael~T. Lacey, and Fan Yang.
\newblock A polynomial {R}oth theorem for corners in finite fields.
\newblock {\em Mathematika}, 67(4):885--896, 2021.

\bibitem[KM23]{KM23}
Zander Kelley and Raghu Meka.
\newblock Strong bounds for 3-progressions.
\newblock In {\em 2023 {IEEE} 64th {A}nnual {S}ymposium on {F}oundations of
  {C}omputer {S}cience---{FOCS} 2023}, pages 933--973. IEEE Computer Soc., Los
  Alamitos, CA, [2023] \copyright 2023.

\bibitem[Len24]{Leng24}
James Leng.
\newblock A quantitative bound for {S}zemer\'edi's theorem for a complexity one
  polynomial progression over {$\Bbb Z/N\Bbb Z$}.
\newblock {\em Discrete Anal.}, pages Paper No. 3, 33, 2024.

\bibitem[LSS23]{LSS23}
James Leng, Ashwin Sah, and Mehtaab Sawhney.
\newblock Improved bounds for five-term arithmetic progressions.
\newblock {\em arXiv preprint arXiv:2312.10776}, 2023.
\newblock Available at \url{https://arxiv.org/pdf/2312.10776}.

\bibitem[LSS24a]{LSS24}
James Leng, Ashwin Sah, and Mehtaab Sawhney.
\newblock Improved bounds for {S}zemer\'{e}di's theorem.
\newblock {\em arXiv preprint arXiv:2402.17995}, 2024.
\newblock Available at \url{https://arxiv.org/pdf/2402.17995}.

\bibitem[LSS24b]{LSS24b}
James Leng, Ashwin Sah, and Mehtaab Sawhney.
\newblock Quasipolynomial bounds on the inverse theorem for the {G}owers
  ${U}^{s+1}[{N}]$-norm.
\newblock {\em arXiv preprint arXiv:2402.17994}, 2024.
\newblock Available at \url{https://arxiv.org/pdf/2402.17994v3}.

\bibitem[Pel18]{Peluse18}
Sarah Peluse.
\newblock Three-term polynomial progressions in subsets of finite fields.
\newblock {\em Israel J. Math.}, 228(1):379--405, 2018.

\bibitem[Pel19]{Peluse19}
Sarah Peluse.
\newblock On the polynomial {S}zemer\'edi theorem in finite fields.
\newblock {\em Duke Math. J.}, 168(5):749--774, 2019.

\bibitem[Pel20]{Peluse20}
Sarah Peluse.
\newblock Bounds for sets with no polynomial progressions.
\newblock {\em Forum Math. Pi}, 8:e16, 55, 2020.

\bibitem[Pel24]{Peluse24}
Sarah Peluse.
\newblock Subsets of {$\mathbb{F}_p^n \times \mathbb{F}_p^n$} without {$\rm
  L$}-shaped configurations.
\newblock {\em Compos. Math.}, 160(1):176--236, 2024.

\bibitem[Pol09]{DHJ09}
DHJ Polymath.
\newblock Density {H}ales--{J}ewett and {M}oser numbers in low dimensions.
\newblock {\em Unpublished, http://michaelnielsen.org/polymath1}, 2009.

\bibitem[Pol12]{DHJ12}
D.~H.~J. Polymath.
\newblock A new proof of the density {H}ales-{J}ewett theorem.
\newblock {\em Ann. of Math. (2)}, 175(3):1283--1327, 2012.

\bibitem[PP22]{PP22}
Sarah Peluse and Sean Prendiville.
\newblock A polylogarithmic bound in the nonlinear {R}oth theorem.
\newblock {\em Int. Math. Res. Not. IMRN}, (8):5658--5684, 2022.

\bibitem[Rot53]{Roth53}
K.~F. Roth.
\newblock On certain sets of integers.
\newblock {\em J. London Math. Soc.}, 28:104--109, 1953.

\bibitem[Shk05]{Shk05}
I.~D. Shkredov.
\newblock On a problem of {G}owers.
\newblock {\em Izvestiya: Mathematics}, 71(1):46--48, 2005.

\bibitem[Shk06]{Shk06}
I.~D. Shkredov.
\newblock On a generalization of {S}zemer\'edi's theorem.
\newblock {\em Proc. London Math. Soc. (3)}, 93(3):723--760, 2006.

\bibitem[Sze75]{Sz75}
E.~Szemer\'{e}di.
\newblock On sets of integers containing no {$k$} elements in arithmetic
  progression.
\newblock In {\em Proceedings of the {I}nternational {C}ongress of
  {M}athematicians ({V}ancouver, {B}.{C}., 1974), {V}ol. 2}, pages 503--505.
  Canad. Math. Congr., Montreal, QC, 1975.

\bibitem[VdW27]{vdw27}
Bartel~Leendert Van~der Waerden.
\newblock Beweis einer baudetschen vermutung.
\newblock {\em Nieuw Arch. Wiskunde}, 15:212--216, 1927.

\end{thebibliography}

\end{document}